\documentclass[11p,a4paper]{article}

\usepackage{amsmath,amssymb,amsthm}
\usepackage{ascmac}
\usepackage{geometry}
\usepackage{bm}
\usepackage{natbib}
\usepackage[pdftex]{graphicx}
\usepackage{here}
\usepackage{indentfirst}

\newtheorem{theorem}{Theorem}
\newtheorem{corollary}{Corollary}

\newtheorem{proposition}{Proposition}
\newtheorem{assumption}{Assumption}
\theoremstyle{definition}

\renewcommand{\tilde}{\widetilde}

\DeclareMathOperator{\Var}{Var}

\begin{document}

\title{Power variations and testing for co-jumps: the small noise approach\thanks{This version: \today.}}

\author{Daisuke Kurisu\thanks{Graduate student, Graduate School of Economics, University of Tokyo, 7-3-1 Hongo, Bunkyo-ku, Tokyo 113-0033, Japan. E-mail: dkurisu.mathstat@gmail.com}}

\date{}

\maketitle

\begin{abstract}
In this paper we  study the effects of noise on bipower variation (BPV), realized volatility (RV) and testing for co-jumps in high-frequency data under the\textit{ small} noise framework. We first establish asymptotic properties of BPV in this framework.  In the presence of  the small noise, RV is asymptotically biased and the additional asymptotic conditional variance term appears in its limit distribution. 
We also propose consistent estimator for the asymptotic conditional variances of RV.  Second, we derive the asymptotic distribution of the test statistic proposed in \cite{JT(2009)} under the presence of small noise for testing the presence of co-jumps  in a two dimensional It\^o semimartingale. In contrast to the setting in \cite{JT(2009)}, we show that the additional conditional asymptotic variance terms appear, and propose consistent estimator for the asymptotic  conditional  variances in order to make the test feasible. Simulation experiments show that our asymptotic results give reasonable approximations in the finite sample cases.   
\end{abstract}


\section{Introduction}

Recently,
a considerable interest has been paid on estimation and testing
for underlying continuous time stochastic processes based on  high-frequency financial data.
In the analysis of high-frequency data,
it is important to take into account the influence of market microstructure noise. 
The market microstructure noise captures a variety of frictions inherent in trading processes such as bid-ask bounce and discreteness of price changes 
\citep{AY(2009)}. 
There are a large number of 
papers on the high-frequency data analysis in the presence of noise. For example, \cite{ZMA(2005)}, \cite{BR(2006)} and \cite{BR(2014)} 
investigate the case when the log-price follows a 
diffusion process  observed with an  additive noise. They assume that 
the size of noise dose not depend on the observation frequency. To be precise, they assume that observed log-prices are of the form
\begin{align}\label{E0}
Y_{t_{i}^{n}}^{(m)} &= X_{t_{i}^{n}}^{(m)} + v_{i}^{(m)},\quad i = 1, \hdots, n,\quad m=1,\hdots,d,
\end{align}
where $(X_{t} = (X_{t}^{(1)}, \hdots, X_{t}^{(d)})^{\top})_{0 \leq t \leq 1}$ is the underlying $d$-dimensional continuous time log-price process and  $v_i = (v_{i}^{(1)},\hdots, v_{i}^{(d)})^{\top}$ are $d$-dimensional i.i.d. random noise vectors  of which each component has mean $0$ and constant variance  
independent of the process $(X_{t})_{0 \leq t \leq 1}$. 
\cite{ZMA(2005)} and \cite{BR(2006)} study the one dimensional ($d=1$) case,  
and \cite{BR(2014)} 
study the 
multi-dimensional case.   Intuitively, the assumption of the constant noise variance means that the noise 
is dominant to the log-price when the observation frequency increases. 
In this paper, we instead assume that the effect of noise depends on the frequency of the observation, and the observed log-prices are of the form
\begin{align}\label{E} 
Y_{t_{i}^{n}}^{(m)} &= X_{t_{i}^{n}}^{(m)} + \epsilon_{n,m}v_{i}^{(m)},\quad i = 1, \hdots, n, \quad m=1,\hdots ,d,
\end{align}
where $\epsilon_{n} = (\epsilon_{n,1}, \hdots , \epsilon_{n,d})^{\top}$ is a $d$-dimensional nonstochastic sequence satisfying $\epsilon_{n,m} \downarrow 0$ as $n \to \infty$ for each $m$. We call this assumption {\it small noise}. 
Under the small noise assumption, the noise 
is vanishing as the observation frequency increases. Hence the small noise assumption is interpreted as an intermediate assumption between the no noise assumption and the constant noise variance assumption. Related literature that considers small noise includes \cite{GJ(2001a)}, \cite{BHLS(2008)}, \cite{LXZ(2014)}, \cite{LZL(2015)} and among others. \cite{HL(2006)} give an empirical evidence that the market microstructure noise is small. 

The first purpose of the paper is to investigate the effect of small noise on bipower variation (BPV) proposed in \cite{BS(2004), BS(2006)} and 
the estimation of the integrated volatility. We establish the asymptotic properties of BPV  when the latent process $(X_{t})_{0\leq t \leq 1}$ is a one dimensional It\^o semimaringale. We also propose procedures to estimate integrated volatility using realized volatility (RV) and the asymptotic conditional variances which appear in the limit distribution of RV under the small noise assumption. In contrast to the no noise model, RV is asymptotically biased and an additional asymptotic conditional variance term appears in its limit distribution (see also \cite{BR(2006)}, \cite{HL(2006)}, \cite{LZL(2015)}, \cite{KK(2015)} and amang others). In the recent related literature, \cite{LZL(2015)} proposed the unified approach for estimating the integrated volatility of a diffusion process when both small noise and asymptotically vanishing rounding error are present. In this paper, we only consider the additive noise but we assume that the log-price process $(X_{t})_{0 \leq t \leq 1}$ is a $d$-dimensional It\^o semimartingale which includes a diffusion process as a special case.

The second purpose of this paper is to propose a procedure to test the existence of co-jumps in two log-prices when the small noise exists. Examining whether two asset prices have contemporaneous jumps (co-jumps) or not is one of the 
effective approaches toward distinguishing  between systematic and idiosyncratic jumps of asset prices and also important in option pricing and risk management. From the empirical side, \cite{GST(2014)} investigate co-jumps and give a strong evidence for the presence of co-jumps. \cite{BTL(2013)} provide another empirical evidence for the dependence in the extreme tails of the distribution governing jumps of two stock prices. In spite of the importance of 
this problem, a testing procedure for co-jumps  is not sufficiently studied. \cite{JT(2009)} is the seminal paper in this literature and other important contributions include \cite{GM(2012)} and \cite{BW(2015)}. \cite{GM(2012)} study the no noise model. \cite{BW(2015)} is a recent important contribution to co-jump test for the model (\ref{E0}). Their co-jump test is based on the wild bootstrap-type approach and for testing the null hypothesis that observed two log-prices  have no co-jumps. 
On the other hand, a grate variety of testing methods for detecting the presence of jumps in the one dimensional case have been developed. See, for example, \cite{BS(2006)}, \cite{FW(2007)}, \cite{JO(2008)}, \cite{B(2008)}, \cite{J(2008)}, \cite{M(2009)}, and \cite{AJ(2009)} for the  no noise model,  and \cite{AJL(2012)} and \cite{L(2013)} for the model (\ref{E0}) with, conditionally on $X$, mutually independent noise. 

Our idea of estimating integrated volatility and the asymptotic conditional variance of RV is based on the SIML method developed in  \cite{KS(2013)} for  correcting the bias of RV and the  truncation method developed  in \cite{M(2009)}. For a construction of a co-jump test, we assume that the process $(X_{t})_{0 \leq t \leq 1}$ in the model (\ref{E}) is a two dimensional It\^o semimartingale and investigate the asymptotic properties of the test statistic proposed in \cite{JT(2009)}. We show that, because of  the presence of the small noise, the asymptotic distribution of the test statistic is different from their result. In fact, the additional asymptotic conditional variance appears in its limit distribution. We develop a fully data-driven procedure to estimate the asymptotic variance of the test statistics based on similar technique used in the estimation of integrated volatility and the asymptotic variance of RV.

The numerical experiments show that our proposed method gives a good approximation of the limit distribution of RV and reasonable result for the estimation of integrated volatility. Our proposed testing procedure of co-jumps also improves the empirical size  in the presence of noise compared with the test in \cite{JT(2009)}.

This paper is organized as follows. In Section 2 we describe the theoretical settings
of the underlying It\^o semimartingale and market microstructure noise. In Section 3 we investigate the effects of noise on the asymptotic properties of BPV, and give some comments on the stable limit theorems of RV. We also propose an estimation method of the integrated volatility in the {\it small noise} framework. In Section 4 we study statistics related to the detection of co-jumps in the two dimensional setting when the noise satisfy the small noise assumption. Then we propose a testing procedure for the presence of co-jumps. 
In Section 5 we give estimation methods of asymptotic conditional variances which appear in the limit theorems of RV and co-jump test statistic studied in the previous sections. We report some simulation results in Section 6 and we give some concluding remarks in Section 7. Proofs are collected in Appendix A.

\section{Setting}

We
consider a continuous-time financial
market in a fixed terminal time $T$. We set $T=1$ without loss of generality. 
The underlying log-price is a $d$-dimensional It\^o semimartingale. 
We observe the log-price process in high-frequency contaminated by the market microstructure noise.

Let
the first filtered probability space be
$(\Omega^{(0)}, \mathcal{F}^{(0)}, (\mathcal{F}_t^{(0)})_{t\geq 0}, P^{(0)})$
on which the $d$-dimensional It\^o semimartingale
$(X_{t})_{0\leq t\leq1}$ is defined, and let the second filtered probability space be\\
$(\Omega^{(1)}, \mathcal{F}^{(1)}, (\mathcal{F}_t^{(1)})_{t\geq 0}, P^{(1)})$
on which the market microstructure noise terms $v_{t_i^n}$ are defined for the discrete time points $0\leq t_{1}^{n} < \cdots <t_n^n\leq 1$.
Then
we consider the filtered probability space
$(\Omega, \mathcal{F}, (\mathcal{F}_t)_{t\geq 0}, P) \;,$
where
$\Omega=\Omega^{(0)}\times \Omega^{(1)} ,$
$\mathcal{F}=\mathcal{F}^{(0)}\otimes \mathcal{F}^{(1)} $ and 
$\mathcal{F}_t = {\bigcap}_{s>t} \mathcal{F}_s^{(0)}\otimes \mathcal{F}_s^{(1)} $ for $0\leq t\leq s\leq 1$.

We
consider the following model
for the observed log-price at $t_i^n\in [0,1] $
as

\begin{align}
  Y_{t_i^n}^{(m)}=X_{t_i^n}^{(m)}+ \epsilon_{n,m} v_{i}^{(m)} \quad i=1,\cdots,n, \quad m=1,\hdots, d,\nonumber 
\end{align}
where $v_n = (v_{n}^{(1)},\hdots, v_{n}^{(d)})^{\top}$ are $d$-dimensional  i.i.d. random noise and noise coefficient $\epsilon_{n} = (\epsilon_{1,n},\hdots, \epsilon_{d,n})^{\top}$ is a sequence of $d$-dimensional vector which depends on sample size $n$ and for each $m$, $\epsilon_{m,n} \to 0$. We assume that these terms satisfy Assumptions 2 and 3 described below. Moreover,  let $X = (X_{t})_{0\leq t\leq 1}$ be an It\^o semimartingale of the form
\begin{align}\label{E1}
X_{t}&=X_{0}+\int_0^t b_s ds+\int_0^t \sigma_s dW_s
 + \int_s\int_{\mathbb{R}^{d}} \kappa \circ \delta (s,x)(\mu - \nu)(ds,dx)
\nonumber \\
 &\qquad + \int_s\int_{\mathbb{R}^{d}} \kappa' \circ\delta (s,x) \mu (ds,dx)\;,
\end{align}
where $(W_s)$ is a $d'$-dimensional standard Brownian motion, $(b_{s})$ is a $d$-dimensional adapted process, 
$\sigma_s$ is a $(d \times d')$-(instantaneous) predictable volatility process and we define the process $c = \sigma \sigma^{\top}$.
Furthermore, $\delta (\omega ,s, x)$ is a predictable function on $\Omega \times \mathbb{R}_{+} \times \mathbb{R}^{d}$, $\kappa: \mathbb{R}^{d} \to \mathbb{R}^{d}$ is a continuous truncation function with compact support and $\kappa'(x) = x-\kappa(x)$,
$\mu (\cdot )$ is a Poisson random measure on $\mathbb{R}_{+}\times \mathbb{R}^{d}$ and
$\nu (ds,dz) = ds\otimes \lambda(dz)$ is a predictable compensator or intensity measure of $\mu$. We partially follow the notation used in \cite{JP(2012)}.
We assume that the observed times 
$0=t_0^n<t_1^n <\cdots <t_n^n=1$
are such that $t_i^n-t_{i-1}^n=1/n =\Delta_n$.
When $d=2$ (bivariate case), let
\begin{align}
\Delta X_{t} &= X_{t} - X_{t-},\quad \tau = \inf\{t: \Delta X_{t}^{(1)}\Delta X_{t}^{(2)} \neq 0\},  \nonumber \\
\widetilde{\Gamma} &= \{(\omega, t, x):\delta^{1}(\omega, t,x)\delta^{2}(\omega, t, x) \neq 0\}, \nonumber 
\end{align}
and for $i=1,2$, define
\begin{align}
\delta^{'i}_{t}(\omega) &= 
\begin{cases}
\int_{\mathbb{R}^{d}}(\kappa^{i} \circ \delta 1_{\widetilde{\Gamma}})(\omega, t,x)\lambda(dx) & \text{if the integral makes sence,} \nonumber \\
+\infty & \text{otherwise.}
\end{cases}
\end{align}

We also make the following assumptions.  
\begin{assumption} 

\begin{itemize}
\item[(a)] The path $t \mapsto b_{t}(\omega)$ is locally bounded.
\item[(b)] The process $\sigma$ is  continuous.
\item[(c)] We have $\sup_{\omega,x}||\delta(\omega, t,x)|| /\gamma(x)$ is locally bounded for a deterministic nonnegative function  satisfying 
$\int_{\mathbb{R}^{d}}(\gamma(x)^{h}\wedge 1)\lambda(dx) < \infty$. for some $h \in (0,2)$.
\item[(d)] For each $\omega$, and $i=1,2$, the path $t \mapsto \delta^{'i}_{t}(\omega)$ is locally bounded on the interval $[0,\tau(\omega))$.
\item[(e)] We have $\int_{t}^{t+u}||\sigma_{s}||ds > 0$ a.s. for all $t,u>0$.
\end{itemize}
\end{assumption}

If $X$ does not have the last two terms of the right hand side of (\ref{E1}) (these are jump parts of $X$), then we say that $X$ is continuous. Otherwise, we say that $X$ is discontinuous. For the noise term, we assume the following conditions.
\begin{assumption}
There exist some $q\geq 0$ and $\zeta_{m}>0$, $1 \leq m \leq d$ such that
\begin{align}
n\epsilon_{m,n}^{2} &= \zeta_{m} n^{-2q} + O\left(n^{-(1+2q)}\right). \nonumber 
\end{align}
\end{assumption}

\begin{assumption}
$\{v_i\}_{i=1}^{\infty}$ is a sequence of i.i.d. $d$-dimensional standard normal random variables.

\end{assumption}

When $q =0$, Assumption 2 coincides with the small noise assumption in \cite{KK(2015)}. 
If the noise coefficient does not depend on the sampling scheme, that is, for each component there exist some positive constants $\epsilon_{m}$ such that $\Var(\epsilon_{m,n}v_{1}^{(m)})  = \epsilon_{m}^{2}$ in the model (\ref{E}), then the effect of noise is asymptotically dominant. This case corresponds to the assumption that the variance of noise is constant. 
Assumption 2 means that the effect of noise depends on a sample number $n$. Hence the effect of noise gets smaller if the observation frequency increases.

\section{The Effects of Small Noise on BPV and RV}

 In this section, we assume that the process $(X_{t})_{0 \leq t \leq 1}$ is one dimensional ($d=1$), and give the asymptotic properties of BPV and give some remarks on the problem of an estimation of integrated volatility and the asymptotic conditional variance of RV under the presence of small noise. 

\subsection{Asymptotic Properties of BPV and RV}

Bipower variation (BPV) and realized volatility (RV) are often used for estimating integrated volatility. We give some results on asymptotic properties of BPV and RV. 
Let $\Delta_{i}^{n}X = X_{i\Delta_{n}}- X_{(i-1)\Delta_{n}}$ and define the following statistics:
\begin{align}
\widetilde{V}_{p}^{(n)}(X) &= \sum_{i=1}^{n}|\Delta_{i}^{n}X|^{p}, \nonumber  \\
\widetilde{V}_{r,s}^{(n)}(X) &= \sum_{i=1}^{n-1}|\Delta_{i}^{n}X|^{r}|\Delta_{i+1}^{n}X|^{s}. \nonumber 
\end{align}

According to the above definition, $\widetilde{V}_{2}^{(n)}(X)$ is the realized volatility ($\text{RV}(X)$) and $\widetilde{V}_{1,1}^{(n)}(X)$ is the bipower variation ($\text{BPV}(X)$). First we give asymptotic properties of above statistics. The following result describes the effect of small noise. 

\begin{proposition}\label{Prp1}

Suppose Assumptions 1, 2 and 3 are satisfied. Let $r,s$ and $k$ be positive integers, then
\begin{align}
\widetilde{V}_{2r}^{(n)}(Y) - \widetilde{V}_{2r}^{(n)}(X) &=  
O_{P}(n^{1-(r+ q) }), \nonumber \\ 
\widetilde{V}_{r,s}^{(n)}(Y) - \widetilde{V}_{r,s}^{(n)}(X) &= O_{P}(n^{1-(r + s)/2 - q}) . \nonumber 
\end{align}


\end{proposition}

In the following results, we freely use the stable convergence arguments and $\mathcal{F}^{(0)}$-conditionally Gaussianity, which have been developed and explained by \cite{J(2008)} and \cite{JP(2012)}, and use the notation $\stackrel{\mathcal{L}-s}{\longrightarrow}$ as stable convergence in law. For the general reference on stable convergence, we also refer to \cite{HL(2015)}. The following proposition describes the case when the effect of noise does not matter asymptotic properties of BPV. The result follows immediately from Proposition 1.

\begin{proposition}\label{Prp11}

Suppose Assumptions 1,2 and 3 are satisfied. Let $X$ be continuous, and $r$ and $s$ be positive integers such that $(r+s)/2$ is integral.
If $q>0$, then we have the following convergence in probability:
\begin{align}
n^{(r+s)/2-1}\widetilde{V}_{r,s}^{(n)}(Y) \stackrel{\mathrm{P}}{\longrightarrow}
m_{r}m_{s}\int_{0}^{1}\sigma_{u}^{r+s}du, \nonumber
\end{align}
and  if $q>1/2$, then we have the following stable convergence in law:
\begin{align}
&\sqrt{n}\left(n^{(r+s)/2-1}\widetilde{V}_{r,s}^{(n)}(Y) - m_{r}m_{s}\int_{0}^{1}\sigma_{u}^{r+s}du\right) \nonumber \overset{\mathcal{L}-s}{\longrightarrow} U, 
\end{align}
where $m_{r}=2^{r/2}{\Gamma ((r+1)2^{-1}) / \Gamma(2^{-1})}$ for $r>0$ and  $U$ is $\mathcal{F}^{(0)}$-conditionally Gaussian with mean $0$ and $\mathcal{F}^{(0)}$-conditional variance $E[U^{2}|\mathcal{F}^{(0)}] = (m_{2r}m_{2s} + 2m_{r}m_{s}m_{r+s} -3m_{r}^{2}m_{s}^{2})\int_{0}^{1}\sigma_{u}^{2(r+s)}du$.
\end{proposition}

From the remark of Theorem 2.5 in \cite{BGJPS(2006)}, if $X$ is continuous, then
\begin{align}
\sqrt{n}\left(\text{BPV}(X) - m_{1}^{2}\int_{0}^{1}\sigma_{s}^{2}ds\right) &\overset{\mathcal{L}-s}{\longrightarrow} \widetilde{U}, \nonumber
\end{align}
where $\widetilde{U}$ has the same distribution as $U$ in Proposition \ref{Prp11} (replacing $r,s=1$). The latter part of Proposition \ref{Prp11}  implies that if $q> 1/2$, we can replace $\widetilde{V}_{1,1}^{(n)}(X)$ as $\widetilde{V}_{1,1}^{(n)}(Y)$. In such a case we can use BPV as the consistent estimator of integrated  volatility. 

Next we consider the asymptotic properties of $\text{RV}$. When the underlying process $X$ is continuous, the $\text{RV}(X)$ 
is often used for estimating integrated volatility in the no noise case. In this case, to construct a confidence interval of integrated volatility or construct a jump test proposed in \cite{BS(2006)} for example, we must consistently estimate the asymptotic conditional variance of the limit distribution of RV and jump test statistics. 
In place of BPV, multipower variation (MPV) is often used for estimating volatility functionals $\int_{0}^{1}\sigma_{s}^{p}ds$ for $p \geq 1$ which appear in the limit distribution of RV. When log-price $X$ is contaminated by small noise, by the similar argument in Proposition \ref{Prp1}, it is possible to obtain the asymptotic property for the  special case of MPV:
\begin{align}
&\Delta_{n}^{-1}\sum_{i=1}^{n-3}|\Delta_{i+3}^{n}Y||\Delta_{i+2}^{n}Y||\Delta_{i+1}^{n}Y||\Delta_{i}^{n}Y| \nonumber \\
&\quad = \Delta_{n}^{-1}\sum_{i=1}^{n-3}|\Delta_{i+3}^{n}X||\Delta_{i+2}^{n}X||\Delta_{i+1}^{n}X||\Delta_{i}^{n}X| + O_{P}(n^{-q}). \nonumber 
\end{align}
Therefore, when the asymptotic order of noise is sufficiently higher ($q > 0$), MPV is a consistent estimator of $\int_{0}^{1}\sigma_{s}^{4}ds$. In this case, we can use the same procedure as that of the no noise case.

If small nose satisfy the condition $q=0$ in Assumption 2, the effect of noise cannot be ignored. From the first part of Proposition 1, if $q = 0$ and $r=1$, then $\text{RV}(Y) - \text{RV}(X) = O_{P}(1)$. \cite{KK(2015)} proved that RV is asymptotically biased and derived the following two stable convergence results under Assumptions 1, 2 and 3 with $q=0$ in Assumption 2. If $X$ is continuous, then 
\begin{align} 
\text{RV}(Y)&\stackrel{\mathrm{P}}{\longrightarrow} \int_{0}^{1}\sigma_{s}^{2}ds +2\zeta_{1} \equiv U_{0,1},  \label{RV1} \\
\sqrt{n}(\text{RV}(Y) &- U_{0,1}) \overset{\mathcal{L}-s}{\longrightarrow} U_{1} + U_{2} + U_{3},\nonumber 
\end{align}
where $U_{j}$ for $j=1,2,3$ are  $\mathcal{F}^{(0)}$-conditionally mutually independent  Gaussian random variables with mean $0$ and $\mathcal{F}^{(0)}$-conditional variances 
$E(U^{2}_{1}|\mathcal{F}^{(0)}) = 2\int_{0}^{1}\sigma_{s}^{4}ds$,
$E(U^{2}_{2}|\mathcal{F}^{(0)}) = 8\zeta_{1} \int_{0}^{1}\sigma_{s}^{2}ds$ and 
$E(U^{2}_{3}|\mathcal{F}^{(0)}) = 12\zeta_{1}^{2}$. 

If $X$ is the It\^o semimartingale of the form (\ref{E1}), then
\begin{align}
\text{RV}(Y) &\stackrel{\mathrm{P}}{\longrightarrow} \int_{0}^{1}
\sigma_{s}^{2}ds +\sum_{0\leq s \leq 1}(\Delta X_{s})^{2}+2\zeta_{1} \equiv U_{0,2}, \label{RV2}  \\
\sqrt{n}(\text{RV}(Y)&- U_{0,2}) \overset{\mathcal{L}-s}{\longrightarrow} U_{1} + U_{2}, \nonumber 
\end{align}
where $U_{j}$ for $j=1,2$ are  $\mathcal{F}^{(0)}$-conditionally mutually independent Gaussian random variables with mean $0$ and $\mathcal{F}^{(0)}$-conditional variances 
$E(U^{2}_{1}|\mathcal{F}^{(0)}) = 2\int_{0}^{1}\sigma_{s}^{4}ds + 4\sum_{0 \leq s \leq 1}
\sigma_{s}^{2}(\Delta X_{s})^{2}$ and 
$E(U^{2}_{2}|\mathcal{F}^{(0)}) = 8\zeta_{1}\left[\int_{0}^{1}\sigma_{s}^{2}ds + \sum_{0 \leq s \leq 1}(\Delta X_{s})^{2}\right]$. 

To construct a feasible procedure to estimate RV, we need to estimate the noise parameter $\zeta_{1}$ and the asymptotic conditional variance of its limit distribution. In Section 5, we construct estimators of the asymptotic conditional variances of RV.

\subsection{Estimation of Integrated Volatility under Small Noise}
In the model (\ref{E0}) with constant noise variance, it is well known that the variance of noise 
can be estimated by $(2n)^{-1}\sum_{i=1}^{n}(\Delta_{i}^{n}Y)^{2} \stackrel{\mathrm{P}}{\longrightarrow} \Var(v_{1})$.
However, under the small noise assumption, for example $q=0$, this estimation does not work. In fact for the small noise case,  $(2n)^{-1}\sum_{i=1}^{n}(\Delta_{i}^{n}Y)^{2} \stackrel{\mathrm{P}}{\longrightarrow} 0$ regardless of the value of $\zeta_{1}$. Thus we must consider another procedure for estimating $\zeta_{1}$. The separated information maximum likelihood (SIML) method investigated in \cite{KS(2013)} can be used to estimate $\zeta_{1}$. The SIML estimator is a consistent estimator of quadratic variation $[X,X]$ of the process $X$ under both models (\ref{E0}) and (\ref{E}). Therefore, under Assumptions 1, 2 and 3, we have
\begin{align}
[\widehat{X,X}]_{\mathrm{SIML}} \stackrel{\mathrm{P}}{\longrightarrow}  [X,X] \label{SIML},
\end{align}
where $[\widehat{X,X}]_{\mathrm{SIML}}$ is 
the SIML estimator discussed in Section 5. From (\ref{RV1}), (\ref{RV2}) and (\ref{SIML}), we obtain 
\begin{align}\label{SIML11}
\widehat{\zeta_{1}} = {1 \over 2}\left(\sum_{i=1}^{n}(\Delta_{i}^{n}Y)^{2} - [\widehat{X,X}]_{\mathrm{SIML}}\right) \stackrel{\mathrm{P}}{\longrightarrow} \zeta_{1}.  
\end{align} 
We consider two types of the truncated version RV:
\begin{align}
\text{TRVC} &= \sum_{i=1}^{n}(\Delta_{i}^{n}Y)^{2}1_{\{|\Delta_{i}^{n}Y| \leq \alpha \Delta_{n}^{\theta}\}},\quad \text{TRVJ} = \sum_{i=1}^{n}(\Delta_{i}^{n}Y)^{2}1_{\{|\Delta_{i}^{n}Y| > \alpha \Delta_{n}^{\theta}\}}, \nonumber 
\end{align}
where $\alpha > 0$ and $\theta \in (0,1/2)$. 
When $X$ have jumps, we can estimate the jump part of quadratic variation by TRVJ when Assumptions 1, 2 and 3 are satisfied:
\begin{align}
\text{TRVJ} &\stackrel{\mathrm{P}}{\longrightarrow} \sum_{0 \leq s \leq 1}(\Delta X_{s})^{2}. \nonumber 
\end{align}
Then, integrated volatility (IV) can be estimated consistently:
\begin{proposition}\label{IV}
Suppose Assumptions 1,  2 and 3 are satisfied. Then 
\begin{align}
\widehat{\text{IV}} &= \widehat{[X,X]}_{\text{SIML}} -  \text{TRVJ} \stackrel{\mathrm{P}}{\longrightarrow} \int_{0}^{1}\sigma_{s}^{2}ds. \nonumber 
\end{align}
\end{proposition}
\noindent
Proposition \ref{IV} implies that $\widehat{\text{IV}}$ is robust to jumps and small noise. In particular, if $q=0$ in Assumption 2, then we can rewrite $\widehat{\text{IV}} = \widehat{[X,X]}_{\text{SIML}} -  \text{TRVJ} = \text{TRVC} - 2\widehat{\zeta_{1}}$. Therefore, from the remark in Section 3.1 and (\ref{SIML11}), we can estimate integrated volatility by the bias correction of RV.

In this section we considered only one dimensional case, however, an extension to the multivariate case is straightforward for the estimation of the covolatility, $\int_{0}^{1}c_{s}^{(p,q)}ds$ where $c_{s}^{(p,q)}$ is the $(p,q)$ component of the volatility process $c_{s}$ defined in Section 2.

\section{Co-jump test under small noise}

One of the interesting problems in high-frequency financial econometrics is whether 
two asset prices have co-jumps or not. To the best of our knowledge, none of the existing literature has so far proposed a co-jump test in the small noise framework. 
In this section, we consider two dimensional case $(X_{t} = (X_{t}^{(1)}, X_{t}^{(2)})^{\top})_{0 \leq t \leq 1}$ and propose a testing procedure to detect the existence of co-jumps for discretely observed processes contaminated by small noise. For this purpose, we study the asymptotic property of the following test statistic proposed in \cite{JT(2009)}:
\begin{align}
T^{(n)} &= {S_{2,2,2}^{(n)}(Y) \over S_{1,2,2}^{(n)}(Y)},\quad S_{k,r,s}^{(n)}(Y) = \sum_{i=1}^{[n/k]}(\Delta_{i}^{n} Y^{(k,1)})^{r}(\Delta_{i}^{n}Y^{(k,2)})^{s}, \nonumber 
\end{align}
where , $\Delta_{i}^{n} Y^{(k,l)} = Y^{(l)}_{ik\Delta_{n}}  - Y^{(l)}_{(i-1)k\Delta_{n}}$ for $k \geq 2$, $l=1,2$, and $[x]$ is the integer part of $x \in \mathbb{R}$. 
To describe our result, we first decompose the sample space $\Omega$ into three disjoint sets
\begin{align}
\Omega^{(j)} &= \{\omega : \text{on $[0,1]$ the process $\Delta X_{s}^{(1)}\Delta X_{s}^{(2)}$ is not identically $0$}\},\nonumber \\
\Omega^{(d)}  &=  \{ \omega : \text{on $[0,1]$ the processes $\Delta X_{s}^{(1)}$ and $\Delta X_{s}^{(2)}$ are not identically $0$, } \nonumber \\
   &\qquad \text{ but the process $\Delta X_{s}^{(1)} \Delta X_{s}^{(2)}$ is  } \}, \nonumber \\
\Omega^{(c)} &= \{\omega : \text{on $[0,1]$ $X^{(1)}$ and $X^{(2)}$ is continuous}\}. \nonumber 
\end{align}
We test the null hypothesis $\mathbb{H}_{0}$ that observed two log-prices have co-jumps, that is, we are on $\Omega^{(j)}$ against the alternative hypothesis $\mathbb{H}_{1}$ that observed two log-prices have no co-jumps but each log-price have jumps, that is, we are on $\Omega^{(d)}$. 

We first provide the asymptotic property of $T^{(n)}$ under the null hypothesis. For this purpose, we consider the asymptotic property of $S^{(n)}_{k,2,2}$. Evaluating the discretization error of the process,  in restriction to the set $\Omega^{(j)}$, for $k=1,2$, we have  
\begin{align}\label{S1}
S_{k,2,2}^{(n)}(Y) &= S_{k,2,2}^{(n)}(X) + 2R_{k,2,1}^{(n)}(X) + 2R_{k,1,2}^{(n)}(X) +o_{P}(1/\sqrt{n}),  
\end{align}
where 
\begin{align}
R_{k,r,s}^{(n)}(X) &= \epsilon_{r\wedge s,n}\sum_{i=1}^{[n/k]}(\Delta_{i}^{n}X^{(k,1)})^{r}(\Delta_{i}^{n}X^{(k,2)})^{s}(\Delta v_{i}^{(k, r\wedge s)}), \label{Rkrs}\\
\Delta v_{i}^{(k,l)} &= \sum_{j=k(i-1)+1}^{ki}(v_{j}^{(l)} - v_{j-1}^{(l)}) = v_{ki}^{(l)} - v_{k(i-1)}^{(l)}. \label{vk}
\end{align} 
Then it suffices to evaluate the three terms $S_{k,2,2}^{(n)}(X)$, $R_{k,2,1}^{(n)}(X)$ and $R_{k,1,2}^{(n)}(X)$. We also have, in restriction to the set $\Omega^{(j)}$,
\begin{align}
S_{k,2,2}^{(n)}(X) &\stackrel{\mathrm{P}}{\to} \sum_{0\leq s \leq 1}(\Delta X_{s}^{(1)})^{2}(\Delta X_{s}^{(2)})^{2} \equiv S_{0}, \nonumber 
\end{align}
and $R_{k,2,1}^{(n)}(X) = R_{k,1,2}^{(n)}(X) = O_{P}(1/\sqrt{n})$. Finally we obtain the joint stable convergence of these three terms:
\begin{align}
\sqrt{n}(S_{k,2,2}^{(n)}(X)-S_{0}, R_{k,2,1}^{(n)}(X), R_{k,1,2}^{(n)}(X))&\stackrel{\mathcal{L}-s}{\longrightarrow} (U^{(1,2)}_{1,k}, U^{(1,2)}_{2}, U^{(1,2)}_{3}), \nonumber 
\end{align}
where $U^{(1,2)}_{1,k}$, $U_{l}^{(1,2)}$ for $l=2,3$ are $\mathcal{F}^{(0)}$-conditionally mutually independent Gaussian random variables with mean $0$ and the following variances: 
\begin{align}
E[(U_{1,k}^{(1,2)})^{2}|\mathcal{F}^{(0)}] &=  kF_{2,2}, \nonumber \\ 
E[(U_{2}^{(1,2)})^{2}|\mathcal{F}^{(0)}] &= 8\zeta_{2} \sum_{0\leq s\leq1}(\Delta X_{s}^{(1)})^{4}(\Delta X_{s}^{(2)})^{2}, \nonumber \\
E[(U_{3}^{(1,2)})^{2}|\mathcal{F}^{(0)}] &= 8\zeta_{1} \sum_{0\leq s\leq1}(\Delta X_{s}^{(1)})^{2}(\Delta X_{s}^{(2)})^{4}, \nonumber  
\end{align}
 and where
\begin{align}
F_{r,s} &= \sum_{0\leq u\leq1}(r^{2}c_{u}^{(1,1)}(\Delta X_{s}^{(1)})^{2(r-1)}(\Delta X_{u}^{(2)})^{2s} + 2rsc_{u}^{(1,2)}(\Delta X_{u}^{(1)})^{2r-1}(\Delta X_{u}^{(2)})^{2s-1} \nonumber \\
&\qquad \qquad + s^{2}c_{u}^{(2,2)}(\Delta X_{u}^{(1)})^{2r}(\Delta X_{u}^{(2)})^{2(s-1)}). \nonumber 
\end{align}
See the proof of Theorem 1 in Appendix A for details. Therefore, we obtain the following theorem which plays an important role in the construction of our co-jump test.


\begin{theorem}

Suppose Assumptions 1, 2 and 3 are satisfied with $q=0$ in Assumption 2. Let $S_{0} = \sum_{0\leq s \leq 1}(\Delta X_{s}^{(1)})^{2}(\Delta X_{s}^{(2)})^{2}$ and $k=1$ or $2$. Then, in restriction to $\Omega^{(j)}$, we have
\begin{align}
\sqrt{n}(S_{k,2,2}^{(n)}(Y) - S_{0}) &\overset{\mathcal{L}-s}{\longrightarrow} \tilde{U}^{(1,2)} = U_{1,k}^{(1,2)} + U_{2}^{(1,2)} + U_{3}^{(1,2)}. \nonumber 
\end{align}
\end{theorem}

Now we propose a co-jump test for two-dimensional high-frequency data under the presence of small noise. From Theorem 1, we have $T^{(n)} \stackrel{\mathrm{P}}{\to} 1$ under the null hypothesis. Since
\begin{align}
T^{(n)} - 1 &= {S_{2,2,2}^{(n)}(Y) - S_{1,2,2}^{(n)}(Y) \over S_{1,2,2}^{(n)}(Y)}, \nonumber 
\end{align}
we must have the asymptotic distribution of $S_{2,2,2}^{(n)}(Y) - S_{1,2,2}^{(n)}(Y)$. Considering the decomposition of $S^{(n)}_{k,2,2}$ in (\ref{S1}), the asymptotic distribution can be derived from the joint limit distribution of $(S^{(n)}, R^{(n)}_{1}, R^{(n)}_{2})$, where $S^{(n)} = S^{(n)}_{2,2,2}(Y)-S^{(n)}_{1,2,2}(Y)$, $R^{(n)}_{1} = R^{(n)}_{2,2,1}(X) - R^{(n)}_{1,2,1}(X)$ and $R^{(n)}_{2} = R^{(n)}_{2,1,2}(X) - R^{(n)}_{1,1,2}(X)$. 
In restriction to the set $\Omega^{(j)}$, we have
\begin{align}
\sqrt{n}(S^{(n)}, R^{(n)}_{1}, R^{(n)}_{2})&\overset{\mathcal{L}-s}{\longrightarrow} (U_{1,1}^{(1,2)}, U^{(1,2)}_{2}, U^{(1,2)}_{3}). \nonumber 
\end{align}
See the proof of Theorem 2 in Appendix A for details. Hence we obtain the asymptotic property of the test statistics under the null hypothesis. 

\begin{theorem}\label{T_Asy}

Suppose Assumptions 1, 2 and 3 are satisfied with $q=0$ in Assumption 2.  Then under the null hypothesis $\mathbb{H}_{0}$, we have
\begin{align}
\sqrt{n}(T^{(n)} - 1) \overset{\mathcal{L}-s}{\longrightarrow} U = {U_{1,1}^{(1,2)} + U_{2}^{(1,2)} + U_{3}^{(1,2)} \over (U_{0}^{(1,2)})^{2}},  \nonumber
\end{align}
where $U_{0}^{(1,2)} = \sum_{0\leq s\leq1}(\Delta X_{s}^{(1)})^{2}(\Delta X_{s}^{(2)})^{2}$, and  $U^{(1,2)}_{1,1}$, $U_{l}^{(1,2)}$ for $l=2,3$ are $\mathcal{F}^{(0)}$-conditionally mutually independent Gaussian random variables appearing in Theorem 1.
\end{theorem}

We notice that in contrast to Theorem 4.1 in \cite{JT(2009)}, additional asymptotic variance terms appear in the limit distribution of the test statistics that must be estimated for the construction of a co-jump test under the small noise assumption.  If $\zeta_{1} = \zeta_{2}=0$ (no noise case), the result in Theorem 2 corresponds to their result. Hence, Theorem 2 includes their result as a special case.  We propose methods to estimate the asymptotic variance of the test statistic in Section 5. 

Next we consider the asymptotic property of the test statistic under the alternative hypothesis $\mathbb{H}_{1}$. In restriction to the set $\Omega^{(d)}$, because of the presence of small noise, we have
\begin{align}
S_{k,2,2}^{(n)}(Y) &= S_{k,2,2}^{(n)}(X) + 2R_{k,2,1}^{(n)}(X) + 2R_{k,1,2}^{(n)}(X) +o_{P}(1/n). \nonumber 
\end{align}
It is possible to obtain a joint limit theorem of $n\times (S_{2,2,2}^{(n)}(X), S_{1,2,2}^{(n)}(X), \widetilde{R}^{(n)}_{1}, \widetilde{R}^{(n)}_{2})$, where $\widetilde{R}^{(n)}_{1} = R_{2,2,1}^{(n)}(X) + R_{2,1,2}^{(n)}(X)$ and $\widetilde{R}^{(n)}_{2} = R_{1,2,1}^{(n)}(X) + R_{1,1,2}^{(n)}(X)$ in restriction to the set $\Omega^{(d)}$. This yields  
\[
T^{(n)} \stackrel{\mathcal{L}-s}{\longrightarrow} \Phi,
\]
where $\Phi$ is a random variable which is almost surely different from 1. 
The detailed description of $\Phi$ and proof is given in Appendix A. Then we obtain the following proposition.
\begin{proposition}\label{T_SP}
Suppose Assumptions 1, 2 and 3 are satisfied with $q=0$ in Assumption 2.  Then the test statistic $T^{(n)}$ has asymptotic size $\alpha$ under the null hypothesis $\mathbb{H}_{0}$ and is consistent under the alternative hypothesis $\mathbb{H}_{1}$, that is, 
\begin{align}
\mathrm{P}\left(\left.\left|{\sqrt{n}(T^{(n)}-1) \over \sqrt{V^{(j)}_{1,2}}}\right| \geq q_{1- \alpha/2}\right| \mathbb{H}_{0}\right) &\longrightarrow \alpha, \quad \text{if $\mathrm{P}(\Omega^{(j)}) >0$}, \nonumber  \\
\mathrm{P}\left(\left.\left|{\sqrt{n}(T^{(n)}-1) \over \sqrt{V^{(j)}_{1,2}}}\right| \geq q_{1-\alpha/2}\right| \mathbb{H}_{1}\right) &\longrightarrow 1,\quad \text{if $\mathrm{P}(\Omega^{(d)})>0$}, \nonumber 
\end{align}
where $V^{(j)}_{1,2}$ is the asymptotic conditional variance of the random variable given in Theorem 2, $\mathrm{P}(\cdot |\mathbb{H}_{0})$ and $\mathrm{P}(\cdot |\mathbb{H}_{1})$ are conditional probabilities with respect to the sets $\Omega^{(j)}$ and $\Omega^{(d)}$, and $q_{\alpha}$ be the $\alpha$-quantile of standard normal distribution. 
\end{proposition}

Proposition \ref{T_SP} implies that if we have a consistent estimator of $V^{(j)}_{1,2}$,  then we can carry out the co-jump test. The procedures to estimate the asymptotic variance of test satistics are discussed in Section 5. 

\section{Consistent estimation of the asymptotic conditional variances}

In this section, we first construct an estimator of the noise parameter $\zeta_{m}$. Then we propose estimation procedures of the asymptotic conditional variance of RV and the co-jump test statistics. 

\subsection{Estimation of the noise variances} 

The most important characteristic of the SIML is its simplicity compared with the pre-averaging method in \cite{JLMPV(2009)} and the spectral method in \cite{BW(2015)} for estimating the quadratic variation consistently, and have asymptotic robustness for the rounding-error \citep{KS(2010), KS(2011)}. It is also quite easy to deal with the multivariate high-frequency data in this approach as demonstrated in \cite{KS(2011)}.

Let $W_{n} = (Y_{\Delta_{n}}^{\top}, Y_{2\Delta_{n}}^{\top},\hdots,Y_{1}^{\top})^{\top}$ be a $n \times d$ matrix  where $Y_{j\Delta_{n}}$ is the $j$th observation of the process $Y$ and $\mathbf{C}_{n}$ be $n \times n$ matrix,
\begin{align}
\mathbf{C}_{n} &= 
\left(
\begin{array}{ccccc}
1 & 0 & \cdots & 0 & 0 \\
1 & 1 & 0 & \cdots & 0 \\
1 & 1 & 1 & \cdots & 0 \\
1 & \cdots & 1 & 1 & 0 \\
1 & \cdots & 1 & 1 & 1 \\
\end{array}
\right),\quad 
\mathbf{C}_{n}^{-1} = 
\left(
\begin{array}{ccccc}
1 & 0 & \cdots & 0 & 0 \\
-1 & 1 & 0 & \cdots & 0 \\
0 & -1 & 1 & \cdots & 0 \\
0 & \cdots & -1 & 1 & 0 \\
0 & \cdots & 0 & -1 & 1 \\
\end{array}
\right) . \nonumber 
\end{align}
We consider the spectral decomposition of $\mathbf{C}_{n}^{-1}(\mathbf{C}_{n}^{-1})^{\top}$, that is,
\[
\mathbf{C}_{n}^{-1}(\mathbf{C}_{n}^{-1})^{\top} = \mathbf{P}_{n}\mathbf{D}_{n}\mathbf{P}_{n}^{\top} = 2\mathbf{I}_{n} - 2\mathbf{A}_{n},
\]
where $\mathbf{D}_{n}$ is a diagonal matrix with the $k$th element
\[
d_{k} = 2\left[1- \cos\left(\pi \left({2k-1 \over 2n +1}\right)\right)\right], \quad k=1,\hdots,n,
\]
and 
\[
\mathbf{P}_{n} = (p_{jk}),\quad p_{jk} = \sqrt{{2 \over n+(1/2)}}\cos{\left[{2\pi \over 2n+1}\left(k-{1\over 2}\right)\left(j - {1\over 2}\right)\right]}, \quad 1 \leq j,k \leq n,
\]
\[
\mathbf{A}_{n} = {1\over 2}
\left(
\begin{array}{ccccc}
1 & 1 & \cdots & 0 & 0 \\
1 & 0 & 1 & \cdots & 0 \\
0 & 1 & 0 & \cdots & 0 \\
0 & \cdots & 1 & 0 & 1 \\
0 & \cdots & 0 & 1 & 0 \\
\end{array}
\right). 
\]
We transform $W_{n}$ to $Z_{n}$ by
\[
Z_{n} = \Delta_{n}^{-1/2}\mathbf{P}_{n}\mathbf{C}_{n}^{-1}(W_{n} - \overline{W}_{0}),
\]
where $\overline{W}_{0} = (Y_{0}^{\top},\hdots,Y_{0}^{\top})^{\top}$. The SIML estimator of the quadratic variation $[X,X]$ is defined by
\begin{align}
[\widehat{X,X}]_{\mathrm{SIML}} = {1 \over m_{n}}\sum_{k=1}^{m_{n}}Z_{n}^{\top}Z_{n},\quad m_{n} = n^{p},\quad 0< p <{1 \over 2}. \label{SIMLp}
\end{align}
From the straightforward extension of Theorem 1 of \cite{KS(2013)}, we obtain
\begin{align}
[\widehat{X,X}]_{\mathrm{SIML}} \stackrel{\mathrm{P}}{\longrightarrow} [X,X] = \int_{0}^{1}c_{s}ds + \sum_{0\leq s \leq 1}(\Delta X_{s})(\Delta X_{s})^{\top}. \label{SIML2}
\end{align}
From (\ref{RV1}), (\ref{RV2}) and (\ref{SIML2}), we obtain the next proposition.

\begin{proposition}\label{Prp2}  
Suppose Assumptions 1, 2 and 3 with $q=0$ in Assumption 2 are satisfied. Then   

\begin{align}
{1 \over 2}\left(\mathrm{RV}(Y^{(m)}) - [\widehat{X,X}]_{\mathrm{SIML}}^{(m,m)}\right) \stackrel{\mathrm{P}}{\longrightarrow} \zeta_{m}, \quad m=1,\hdots,d, \nonumber 
\end{align}
where $[\widehat{X,X}]_{\mathrm{SIML}}^{(m,m)}$ is the $(m,m)$ component of the SIML estimator.
\end{proposition}

The SIML estimator can be regarded as a modification of the standard maximum likelihood (ML) method under the Gaussian process and an extension for the multivariate case of the ML estimation of the univariate diffusion process with market microstrucure noise by \cite{GJ(2001b)} (see \cite{KS(2013)}). In stead of the SIML estimator, we can also use the quasi maximum likelihood estimator studied in \cite{X(2010)} for the noise robust estimation of quadratic variation.

\subsection{Estimation of the asymptotic conditional variances of RV and the test statistic} 

We propose procedures to estimate the the asymptotic conditional variances of co-jump test statistics studied in Section 4, and RV when the process $(X_{t})_{0 \leq t \leq 1}$ has continuous path or discontinuous path.
We introduce some notations. For $l,m,p,q=1,\hdots,d$, 
\begin{align}
D^{(l,m)}_{p,q}(r,s) &= \sum_{0 \leq u \leq 1}c^{(p,q)}_{u}(\Delta X_{u}^{(l)})^{r}(\Delta X_{u}^{(m)})^{s},\quad r,s \geq 1,\nonumber \\
J^{(l,m)}(r,s) &=  \sum_{0 \leq u \leq 1}(\Delta X_{u}^{(l)})^{r}(\Delta X_{u}^{(m)})^{s}, \quad r,s \geq 2,\nonumber \\
A^{(l,m)}(r) &= \int_{0}^{1}(c_{s}^{(l,m)})^{r}ds, \quad r \geq 1. \nonumber 
\end{align}

To estimate the asymptotic conditional variance of the co-jump test statistic, we consider following statistics. For $l,m,p,q = 1,\hdots,d$ and $r,s \geq 2$, 
\begin{align}
\widehat{D}^{(l,m)}_{p,q}(r,s) &= \sum_{i=1}^{n}(\Delta_{i}^{n}Y^{(l)})^{r}(\Delta_{i}^{n}Y^{(m)})^{s} \nonumber \\
 &\qquad \times \left[{1 \over 2k_{n}\Delta_{n}}\sum_{j \in I_{n}(i)}\left[(\Delta_{i}^{n}Y^{(p)})(\Delta_{i}^{n}Y^{(q)})1_{\{||\Delta_{j}^{n}Y|| \leq \alpha \Delta_{n}^{\theta}\}} - 2\delta_{pq}\sqrt{\widehat{\zeta}_{p}\widehat{\zeta}_{q}}\right]\right], \nonumber \\
\widehat{J}^{(l,m)}(r,s) &=  \sum_{i=1}^{n}(\Delta_{i}^{n}Y^{(l)})^{r}(\Delta_{i}^{n}Y^{(m)})^{s}, \nonumber \\
\widehat{F}^{(l,m)}(r,s) &= r^{2}\widehat{D}^{(l,m)}_{l,l}(2(r-1),2s) + 2rs\widehat{D}^{(l,m)}_{l,m}(2r-1,2s-1) + s^{2}\widehat{D}^{(l,m)}_{m,m}(2r,2(s-1)), \nonumber 
\end{align}
where $I_{n}(i) = \{j \in \mathbb{N} : j \neq i, 1 \leq j \leq n, |i - j| \leq k_{n}\}$, $k_{n} \to \infty$, $k_{n}\Delta_{n} \to 0$ as $n \to \infty$, $\alpha>0$ and $\theta \in (0,1/2)$. 

Although we can construct the consistent estimators of the asymptotic variances based on the truncated functionals (see also \cite{M(2009)} and \cite{JT(2009)}) for the no noise case, the truncation method gives us biased estimators since we now assume the presence of the small noise. 
In fact, since the size of noise gets smaller as the observation frequency increases, truncation is not sufficient to distinguish the continuous part of process $X$ with the noise. To be precise, from the similar argument of Theorem 9.3.2 in \cite{JP(2012)}, we have
\begin{align}
&{1 \over 2k_{n}\Delta_{n}}\sum_{j \in I_{n}(i)}(\Delta_{j}^{n}Y^{(p)})(\Delta_{j}^{n}Y^{(q)})1_{\{||\Delta_{j}^{n}Y|| \leq \alpha \Delta_{n}^{\theta}\}}\nonumber \\
&\quad - {1 \over 2k_{n}\Delta_{n}}\sum_{j \in I_{n}(i)}(\Delta_{j}^{n}C^{(p)} + \epsilon_{p,n}\Delta_{j}v^{(p)})(\Delta_{j}^{n}C^{(q)} + \epsilon_{q,n}\Delta_{j}v^{(q)}) \stackrel{\mathrm{P}}{\longrightarrow} 0, \nonumber 
\end{align}
where $C_{t} = \int_{0}^{t}\sigma_{s}dW_{s}$ and $\Delta_{j}^{n}C^{(p)}$ is the $p$th component of  $C_{j\Delta_{n}} - C_{(j-1)\Delta_{n}}$.

Since $E[\Delta_{i}^{n}C^{(p)}\epsilon_{j,n}\Delta_{i}v^{(q)}] = 0$ and 
$\epsilon_{p,n}\epsilon_{q,n}E[\Delta_{i}v^{(p)}\Delta_{i}v^{(q)}] = 
2\delta_{pq}\sqrt{\widehat{\zeta_{p}}\widehat{\zeta_{q}}}\Delta_{n} + o(\Delta_{n})$ for all $1 \leq p,q \leq d$, $\# I_{n}(i) = 2k_{n}$ where $\delta_{pq}$ is a Dirac's delta function, it follows that
\begin{align}
{1 \over 2k_{n}\Delta_{n}}\sum_{j \in I_{n}(i)}\Delta_{j}^{n}C^{(p)}\Delta_{j}^{n}C^{(q)}- c_{i\Delta_{n}}^{(p,q)} \stackrel{\mathrm{P}}{\longrightarrow}  0. \nonumber 
\end{align}
Therefore, we obtain
\begin{align}
{1 \over 2k_{n}\Delta_{n}}\sum_{j \in I_{n}(i)}(\Delta_{j}^{n}Y^{(p)})(\Delta_{j}^{n}Y^{(q)})1_{\{||\Delta_{j}^{n}Y|| \leq \alpha \Delta_{n}^{\theta}\}} - [c_{i\Delta_{n}}^{(p,q)}+2\delta_{pq}\sqrt{\widehat{\zeta_{p}}\widehat{\zeta_{q}}}] \stackrel{\mathrm{P}}{\longrightarrow}  0. \nonumber 
\end{align}
Then, for $\widehat{D}^{(l,m)}_{p,q}(r,s)$, 
\[
{1 \over 2k_{n}\Delta_{n}}\sum_{j \in I_{n}(i)}\left[(\Delta_{i}^{n}Y^{(p)})(\Delta_{i}^{n}Y^{(q)})1_{\{||\Delta_{j}^{n}Y|| \leq \alpha \Delta_{n}^{\theta}\}} - 2\delta_{pq}\sqrt{\widehat{\zeta_{p}}\widehat{\zeta_{q}}}\right]  
\]
is an estimator of the spot volatility $c_{i\Delta_{n}}^{(p,q)}$. 
By the simple extension of Theorem 2 in \cite{KK(2015)}, it can be obtained that  $(\Delta_{i}^{n}Y^{(l)})^{r}(\Delta_{i}^{n}Y^{(m)})^{s}$ is an unbiased estimator of $(\Delta X_{i\Delta_{n}}^{(l)})^{r}(\Delta X_{i\Delta_{n}}^{(m)})^{s}$. 
Hence, for estimating the asymptotic conditional variance of RV, we consider following statistics:
\begin{align}
\widehat{c}_{i}^{(l,m)} &= {1\over 2k_{n}\Delta_{n}}\sum_{i = k_{n}+1}^{n-k_{n}}(\Delta_{i}^{n}Y^{(l)})(\Delta_{i}^{n}Y^{(m)})1_{\{||\Delta_{i}^{n}Y||\leq \alpha \Delta_{n}^{\theta}\}} -2 \delta_{lm}\sqrt{\widehat{\zeta}_{l}\widehat{\zeta}_{m}}, \nonumber \\
\widehat{A}^{(l,m)}(r) &= \Delta_{n} \sum_{i=1}^{n-k_{n}+1}(\widehat{c}^{(l,m)}_{i})^{r}, \nonumber \\
\widehat{D}^{(l,m)}_{p,q}(1,1) &= \sum_{i=1}^{n}\widehat{c}_{i}^{(p,q)}(\Delta_{i}^{n}Y^{(l)})(\Delta_{i}^{n}Y^{(m)})1_{\{||\Delta_{i}^{n}Y|| > \alpha \Delta_{n}^{\theta}\}}. \nonumber 
\end{align}
Then we obtain the next result.
\begin{proposition}\label{Prp3}  
Suppose Assumptions 1, 2 and 3 are satisfied  with $q=0$ in Assumption 2. Then, for $l,m,p,q = 1,\hdots, d$ and $r,s \geq 1$, we have
\begin{align}
\widehat{D}^{(l,m)}_{p,q}(r,s)  &\stackrel{\mathrm{P}}{\longrightarrow} D^{(l,m)}_{p,q}(r,s),\quad \widehat{J}^{(l,m)}(r,s)  \stackrel{\mathrm{P}}{\longrightarrow} J^{(l,m)}(r,s),   \nonumber \\
\widehat{F}^{(l,m)}(r,s)  &\stackrel{\mathrm{P}}{\longrightarrow} F^{(l,m)}(r,s),\quad \widehat{A}^{(l,m)}(r)\stackrel{\mathrm{P}}{\longrightarrow} A^{(l,m)}(r). \nonumber
\end{align}
\end{proposition}

From Propositions \ref{Prp2} and \ref{Prp3}, we can construct the consistent estimator of the  asymptotic variance of co-jump test statistics

\begin{align}
\widehat{V}^{(j)}_{1,2} &= {\widehat{F}^{(1,2)}(2,2) + 8(\widehat{\zeta_{2}}\widehat{J}^{(1,2)}(2,4) + \widehat{\zeta_{1}}\widehat{J}^{(1,2)}(4,2)) \over n\times \widehat{J}^{(1,2)}(2,2)}, \nonumber
\end{align}
and for consistent estimators of the asymptotic variances of RV, 
\begin{align}
\widehat{V}^{(c)} &= (2\widehat{A}^{(1,1)}(4) + 8\widehat{\zeta_{1}}\widehat{A}^{(1,1)}(2) + 12\widehat{\zeta_{1}}^{2})/n, \quad \text{if $X$ is  continuous}, \nonumber \\
\widehat{V}^{(j)} &=  (2\widehat{A}^{(1,1)}(4) + 4\widehat{D}^{(1,1)}_{1,1}(1,1) + 8\widehat{\zeta_{1}}[\widehat{X,X}]^{\text{SIML}}_{(1,1)})/n, \quad \text{if $X$ is  an It\^o semimartingale of the form (\ref{E1})}. \nonumber
\end{align}
Therefore, we obtain the next two central limit theorems.

\begin{corollary}
Suppose Assumptions 1, 2 and 3 are satisfied with $q=0$ in Assumption 2. Then we have the following convergence in law: 
\begin{itemize} 
\item[(i)] If $X$ is continuous, then 
\begin{align}
(\widehat{V}^{(c)})^{-1/2}(\mathrm{RV}(Y) - U_{0}^{(c)}) &\overset{\mathcal{L}}{\longrightarrow} \text{N}(0,1),  \nonumber 
\end{align}
where $U_{0}^{(c)} = \int_{0}^{1}\sigma_{s}^{2}ds + 2\zeta_{1}$.
\item[(ii)] If $X$ is the It\^o semimartingale of the form ($\ref{E1}$), then
\begin{align}
(\widehat{V}^{(j)})^{-1/2}(\mathrm{RV}(Y) - U_{0}^{(j)}) &\overset{\mathcal{L}}{\longrightarrow} \text{N}(0,1), \nonumber 
\end{align}
where $U_{0}^{(j)} = \int_{0}^{1}\sigma_{s}^{2}ds + \sum_{0 \leq s \leq 1}(\Delta X_{s})^{2}+ 2\zeta_{1}$.
\end{itemize}
\end{corollary}

\begin{corollary}
Suppose Assumptions 1, 2 and 3 are satisfied with $q=0$ in Assumption 2. Then, under the null hypothesis $\mathbb{H}_{0}$, we have  
\begin{align}
(\widehat{V}^{(j)}_{1,2})^{-1/2}(T^{(n)}-1) &\overset{\mathcal{L}}{\longrightarrow} \text{N}(0,1). \nonumber
\end{align}
\end{corollary}

\section{Numerical Experiments}

In this section, we report several results of numerical experiment. We simulate a data generating process according to the procedure in \cite{CT(2004)}. First we consider the estimation of integrated volatility. We used following data generating processes:
\begin{align}
dX_{t} &= \sigma_{t}dW_{t} + dJ^{cp}_{t}, \label{CP} \\
dX_{t} &= \sigma_{t}dW_{t} + dJ^{\beta}_{t},  \label{ST}
\end{align}
where $W$ is a standard Brownian motion. $J_{t}^{cp}$ is a compound Poisson process in $t \in [0,1]$ where the intensity of the Poisson process is $\lambda = 10$ and the jump size  is a uniform distribution on $[-0.3,-0.05]\cup[0.05, 0.3]$. $J^{\beta}_{t}$ is a $\beta$-stable process with $\beta=1.5$. We also set the truncate level $\alpha \Delta_{n}^{\theta}$ as $\alpha = 2 $ and $\theta = 0.48$. These values are also used in the simulation of the co-jump test. 
As a stochastic volatility model, we use
\begin{align} \label{EE1}
d\sigma_{t}^{2} &= \alpha(\beta - \sigma_{t}^{2})dt + \sigma_{t}^{2}dW_{t}^{\sigma},
\end{align}
where $\alpha=5$, $\beta=0.2$ and $\rho= E[dW_{t}dW_{t}^{\sigma}] = -0.5$. For the market microstructure noise, we have adopted the three types of  Gaussian noise $\text{N}(0, \zeta \Delta_{n})$ with $\zeta = 0, 10^{-4}$ and $10^{-2}$ (we call these cases as (i), (ii) and (iii)). To estimate $\zeta$, we used SIML with $p= 0.49$ in (\ref{SIMLp}) and we also use this value in the following simulations. We also consider following cases:
\begin{itemize}
\item CJ1: $1/\Delta_{n} = 20,000$, $X$ follows the process (\ref{CP}).
\item CJ2: $1/\Delta_{n} = 30,000$, $X$ follows the process (\ref{CP}).
\item SJ1: $1/\Delta_{n} = 20,000$, $X$ follows the process (\ref{ST}).
\item SJ2: $1/\Delta_{n} = 30,000$, $X$ follows the process (\ref{ST}).
\end{itemize}

CJ and SJ correspond to a finite activity jump case and an infinite activity jump case respectively. We present the root mean square errors (RMSE) for each case in Table 1. Compared with CJ, the RMSEs in SJ tends to be large. This is because of the difficulty in a finite sample case to distinguish infinite activity jumps and the other part of the observed process.   
%
%
%
%
%
%
%
%

Next we check the performance of the proposed feasible CLT of RV. We use following data generating processes:
\begin{align}
dX_{t}  &= \sigma_{t}dW_{t},\quad \text{for continuous case}, \nonumber \\
dX_{t}  &= \sigma_{t}dW_{t} + dJ^{cp}_{t}, \quad \text{for jump case}, \label{E2}    
\end{align} 
where the process (\ref{E2}) and $\sigma$ are the  same processes as (\ref{EE1}) and (\ref{CP}). 
Figures 1 and 2 give the standardized empirical densities when the small noise is $\text{N}(0,10^{-2}\Delta_{n})$. The simulation size is $N= 1,000$ and the number of the observations is $n= 20,000$. The red line is the density of standard normal distribution. The left figure corresponds to the bias-variance corrected case implied by the Corollary 1 and the right figure corresponds to the no correction case. We found that contamination of small noise still has the significant effect on the distribution of the limiting random variables when $\zeta_{m}\neq 0$  in the continuous case. We can see this in Figure 1. However we have a good approximation for the finite sample distribution of statistics if we correct the effect of noise by using the small noise asymptotics.   

Finally we give simulation results of the co-jump test. We simulate a two dimensional It\^o semimaritingale by the following model:


\begin{align}
dX_{t}^{(1)} &= \sigma_{t}^{(1)}dW_{t}^{(1)} + Z_{t}^{(1)}dN_{t}^{(1)} + Z_{t}^{(3)}dN_{t}^{(3)}, \nonumber \\
dX_{t}^{(2)} &= \sigma_{t}^{(2)}dW_{t}^{(2)} + Z_{t}^{(2)}dN_{t}^{(2)} + Z_{t}^{(4)}dN_{t}^{(3)}, \nonumber
\end{align}

where $W = (W^{(1)}, W^{(2)})$ is a two dimensional Brownian motion and $N^{(j)}$ for $j=1,2,3$ are Poisson processes with intensities $\lambda_{j}$ which are mutually independent and also independent of $W$. $Z = (Z^{(1)}, Z^{(2)}, Z^{(3)}, Z^{(4)})$ is the vector of jump sizes, cross-sectionally and serially independently distributed with laws $F_{Z^{(j)}}$. 
In this simulation, we set $\lambda_{1}=\lambda_{2}=5$ and $\lambda_{3}=10$, and jump size distributions are $Z_{t}^{(1)}, Z_{t}^{(2)} \sim \text{N}(0,5^{-2})$, $Z_{t}^{(3)}, Z_{t}^{(4)} \sim \text{N}(0,10^{-2})$. 
For the volatility process $\sigma$, we set
\begin{align}
d(\sigma_{t}^{(j)})^{2} &= \alpha_{j}(\beta_{j} - (\sigma_{t}^{(j)})^{2})dt + (\sigma_{t}^{(j)})^{2}dW_{t}^{\sigma^{(j)}},\quad j=1,2, \nonumber
\end{align}
where $\sigma^{(1)}$ and $\sigma^{(2)}$ are independent and $\alpha_{1}=\alpha_{2} = 5$, $\beta_{1} = 0.2$, $\beta_{2} = 0.15$, $E[dW_{t}^{\sigma^{(j)}}dW_{t}^{(j)}] = \rho_{j}dt$, $\rho_{1} = -0.5$ and $\rho_{2} = -0.4$. For the market microstructure noise, we use the four types of Gaussian noise for each component with mean $0$ and the same variance $\epsilon_{1,n} = \epsilon_{2,n} = \zeta/n$ where $\zeta = 10^{-4}$, $10^{-2}$, $10^{-1}$ and $1$(we call these cases as I, II, III and IV). We consider the 5\% significant level for following cases:
\begin{itemize}
\item C1 : $1/\Delta_{n} =20,000$, co-jump.

\item C2 : $1/\Delta_{n} =30,000$, co-jump.


\item D1 : $1/\Delta_{n} =20,000$, no co-jump.

\item D2 : $1/\Delta_{n} =30,000$, no co-jump.

\end{itemize}

In Figure 4, we plot the empirical distribution  obtained from Corollary 2 in the case C1-IV. It is interesting to see that our proposed testing procedure gives a good approximation of the limit distribution of test statistics even in large noise case ($\zeta=1$). In Table 2 we also give the empirical size and power of the proposed co-jump test and the test proposed in \cite{JT(2009)} (we call this JT test). We find that JT test is sensitive to the small noise. Theorem 2 implies that in the presence of small noise, the critical value of the test is larger than that of no noise case. Hence the critical value of JT test is small compared with that of the proposed test and the empirical size of JT test tends to be larger than the significant level.  In particular, when the effect of noise is large (C1-III, C1-IV, C2-III and C2-IV), we can see the size distortion of the JT test. On the other hand, the empirical size of the proposed test is close to $0.05$. This result shows that we need to correct the asymptotic variance of JT test in the presence of small noise.  Additionally, the empirical power of our proposed test is very close to 1 in each case.

\section{Concluding Remarks}

In this paper, we developed the small noise asymptotic analysis when the size of the market microstructure noise depends on the frequency of the observation. By using this approach, 
we can identify the effects of jumps and noise in high-frequency data analysis. We investigated the asymptotic properties of BPV and RV in one dimensional case in the presence of small noise. We  proposed methods to estimate integrated volatility and the asymptotic conditional variance of RV. As a result, feasible central limit theorems of RV is established under the small noise assumption when the latent process $X$ is an It\^o semimartingale. Our method gives a good approximation of the limiting distributions of the sequence of random variables.

We also proposed a testing procedure of the presence of co-jumps when the two dimensional latent process is observed with small noise for testing the null hypothesis that the observed two latent processes have co-jumps. Our proposed co-jump test is an extension the co-jump test  in \cite{JT(2009)} for the noisy observation setting. Estimators of the asymptotic conditional variance of the test statistics can be  constructed in a simple way. 
We found that the empirical size of the proposed test works well in finite sample. In particular, proposed co-jump test has a good performance even when the effect of noise is large.


\appendix 

\section{Proofs} 

\subsection{Proofs for Section 3} 

Throughout Appendix, $K$ denotes a generic constant which may change from line to line. We use the techniques developed in \cite{JP(2012)}.
We can replace Assumption 1 to the local boundedness assumption below and such a replacement can be established by the localizing procedure provided in \cite{JP(2012)}.

\begin{assumption}

We have Assumption 1 and there are a constant $A$ and a nonnegative function $\Gamma$ on  $\mathbb{R}^{d}$ for all $(\omega, t, x) \in \Omega \times \mathbb{R}_{+} \times \mathbb{R}^{d}$ such that 
\[
\max\{||b_{t}(\omega)||,||\sigma_{t}(\omega)||, ||X_{t}(\omega)||, ||\Gamma(x)||\} \leq A.
\]

\end{assumption}
Under Assumption 4, we can obtain following inequalities for It\^o semimartingale.
First we consider the decomposition of $X$,
\[
X_{t} = X_{0} + B_{t} + C_{t} + J_{t},
\]
where 
\begin{align}
B_{t} &= \int_{0}^{t}\left(b_{s} + \int_{\mathbb{R}^{d}}\kappa \circ \delta (s,x)\lambda(dx)\right)ds,\nonumber \\
C_{t} &= \int_{0}^{t}\sigma_{s}dW_{s},\quad J_{t} = \int_{0}^{t}\delta(s,x)(\mu-\nu)(ds,dx). \nonumber
\end{align}
For all $p \geq 1$, $s,t \geq 0$, we obtain the following inequalities: 
\begin{align}
E(\sup_{0\leq u \leq t}||B_{s+u} - B_{s}||^{p}|\mathcal{F}^{(0)}_{s}) &\leq Kt^{p}, \label{A1} \\
E(\sup_{0\leq u \leq t}||C_{s+u} - C_{s}||^{p}|\mathcal{F}^{(0)}_{s}) &\leq Kt^{p/2},\label{A3} \\
E(\sup_{0\leq u \leq t}||J_{s+u} - J_{s}||^{p}|\mathcal{F}^{(0)}_{s}) &\leq Kt^{{p\over h}\wedge1}. \label{A4} 
\end{align}
See Section 2.1.5 of \cite{JP(2012)} for details of the derivation of these inequalities.

\vspace{.3cm}
\begin{proof}[Proof of Proposition 1] 
First, we have
\begin{align}
\widetilde{V}_{2r}^{(n)}(Y) - \widetilde{V}_{2r}^{(n)}(X)&= \sum_{i=1}^{n}\sum_{p=1}^{2r}\binom{2r}{p}(\Delta_{i}^{n}X)^{2r-p}(\epsilon_{n}\Delta v_{i})^{p}. \nonumber 
\end{align}
By using the inequalities of It\^o semimartingale, it follows that
\begin{align}
E\left[|(\Delta_{i}^{n}X)^{2r-p}(\epsilon_{n}\Delta v_{i})^{p}|\right] \leq K\Delta_{n}^{({1 \over 2} + q)p}\times \Delta_{n}^{(r-{p\over 2})} = K\Delta_{n}^{q+r},\quad i=1,\hdots,n.  \nonumber 
\end{align}
Therefore, we obtain $\widetilde{V}_{2r}^{(n)}(Y) - \widetilde{V}_{2r}^{(n)}(X) = O_{P}(\Delta_{n}^{q+r-1})$. Second, we have
\begin{align}
\widetilde{V}_{r,s}^{(n)}(Y) - \widetilde{V}_{r,s}^{(n)}(X) &= \sum_{i=1}^{n-1}\left(|\Delta_{i}^{n}X + \epsilon_{n}\Delta v_{i}|^{r}|\Delta_{i+1}^{n}X + \epsilon_{n}\Delta v_{i+1}|^{s} - |\Delta_{i}^{n}X|^{r}|\Delta_{i+1}^{n}X|^{s}\right). \nonumber 
\end{align}
By using the inequalities (\ref{A1}), (\ref{A3}) and (\ref{A4}) repeatedly, we have $E[|\Delta_{i}^{n}X/\sqrt{\Delta_{n}}|^{p}] \leq K$ for $p\geq 1$. Then, we have 
\begin{align}
&E\left[|\Delta_{n}^{(r+s)/2}\left(|\Delta_{i}^{n}X + \epsilon_{n}\Delta v_{i}|^{r}|\Delta_{i+1}^{n}X + \epsilon_{n}\Delta v_{i+1}|^{s}  - |\Delta_{i}^{n}X|^{r}|\Delta_{+}^{n}X|^{s}\right)|\right] \nonumber \\
&\quad = E\left[\left|\left(\left|{\Delta_{i}^{n}X \over \sqrt{\Delta_{n}}} + {\epsilon_{n}\Delta v_{i} \over \sqrt{\Delta_{n}}}\right|^{r}\left|{\Delta_{i+1}^{n}X \over \sqrt{\Delta_{n}}} + {\epsilon_{n}\Delta v_{i+1} \over \sqrt{\Delta_{n}}}\right|^{s}  - \left|{\Delta_{i}^{n}X \over \sqrt{\Delta_{n}}}\right|^{r}\left|{\Delta_{i+1}^{n}X \over \sqrt{\Delta_{n}}}\right|^{s}\right) \right|\right]\nonumber \\
&\quad \leq K\epsilon_{n}\Delta_{n}^{-1/2} = O(\Delta_{n}^{q}),\quad i=1,\hdots, n. \nonumber 
\end{align}
Therefore, we obtain $\widetilde{V}_{r,s}^{(n)}(Y) - \widetilde{V}_{r,s}^{(n)}(X) = O_{P}(\Delta_{n}^{q-(r+s)/2+1})$.
\end{proof}

\vspace{.3cm}

\begin{proof}[Proof of Proposition 2]
By Proposition 1,
\begin{align}
n^{(r+s)/2-1}\widetilde{V}^{(n)}_{r,s}(Y) - m_{r}m_{s}\int_{0}^{1}\sigma_{s}^{r+s}ds 
&=n^{(r+s)/2-1}\left(\widetilde{V}^{(n)}_{r,s}(Y) - \widetilde{V}^{(n)}_{r,s}(X)\right) \nonumber \\
&\qquad +\left( n^{(r+s)/2-1}\widetilde{V}^{(n)}_{r,s}(X) - m_{r}m_{s}\int_{0}^{1}\sigma_{s}^{r+s}ds\right) \nonumber \\
& = 
O_{P}\left(n^{-q}\right) + O_{P}\left({1 \over \sqrt{n}}\right). \nonumber 
\end{align}
Therefore, if $q>0$, then $n^{(r+s)/2-1}\widetilde{V}^{(n)}_{r,s}(Y)$ converges in probability to $m_{r}m_{s}\int_{0}^{1}\sigma_{s}^{r+s}ds$. If $q> 1/2$, then from Theorem 2.3 of \cite{BGJPS(2006)}, the first part of above decomposition converges in probability to $0$, and the second part converges stably to $U$ defined in Proposition 2. Then we obtain the desired result. 
\end{proof}

\subsection{Proofs for Section 4} 

\vspace{.3cm}
\begin{proof}[Proof of Theorem 1] 

We decompose $S^{(n)}_{k,2,2}(Y)$ as follows:

\begin{align}
S^{(n)}_{k,2,2}(Y) &= U_{1,n} + \sum_{j=1}^{2}U_{2,j,n} + \sum_{j=1}^{3}U_{3,j,n} + \sum_{j=1}^{3}U_{4,j,n}, \label{S_decomp} 
\end{align}
where 
\begin{align}
U_{1,n} &= \sum_{i=1}^{[n/k]}(\Delta_{i}^{n} X^{(k,1)})^{2}(\Delta_{i}^{n}X^{(k,2)})^{2}, \quad U_{2,1,n} = 2R^{(n)}_{k,2,1}(X),  \quad U_{2,2,n} = 2R^{(n)}_{k,1,2}(X), \nonumber \\
U_{3,1,n} &= \epsilon_{2,n}^{2}\sum_{i=1}^{[n/k]}(\Delta_{i}^{n} X^{(k,1)})^{2}(\Delta v_{i}^{(k,2)})^{2}, \quad U_{3,2,n} = 4\epsilon_{1,n}\epsilon_{2,n}\sum_{i=1}^{[n/k]}(\Delta_{i}^{n} X^{(k,1)})(\Delta_{i}^{n} X^{(k,2)})(\Delta v_{i}^{(k,1)})(\Delta v_{i}^{(k,2)}), \nonumber \\
U_{3,3,n} &= \epsilon_{1,n}^{2}\sum_{i=1}^{[n/k]}(\Delta_{i}^{n} X^{(k,2)})^{2}(\Delta v_{i}^{(k,1)})^{2}, \quad U_{4,1,n} = 2\epsilon_{1,n}\epsilon_{2,n}^{2}\sum_{i=1}^{[n/k]}(\Delta_{i}^{n} X^{(k,1)})(\Delta v_{i}^{(k,1)})(\Delta v_{i}^{(k,2)})^{2},\nonumber \\
U_{4,2,n} &= 2\epsilon_{1,n}^{2}\epsilon_{2,n}\sum_{i=1}^{[n/k]}(\Delta_{i}^{n} X^{(k,2)})(\Delta v_{i}^{(k,1)})^{2}(\Delta v_{i}^{(k,2)}), \quad U_{4,3,n} = \epsilon_{1,n}^{2}\epsilon_{2,n}^{2} \sum_{i=1}^{[n/k]}(\Delta v_{i}^{(k,1)})^{2}(\Delta v_{i}^{(k,2)})^{2}, \nonumber
\end{align}
and where $R_{k,1,1}^{(n)}(X)$, $R_{k,1,2}^{(n)}(X)$, $\Delta v_{i}^{(k,1)}$ and $\Delta v_{i}^{(k,2)}$ are defined in (\ref{Rkrs}) and (\ref{vk}). For the proof of Theorem 1, we first prove the stable convergence of $U_{1,n}$, $U_{2,1,n}$ and $U_{2,2,n}$. Then we prove the joint convergence of these terms. Finally, we will prove that $\sum_{j=1}^{3}U_{3,j,n}$ and $\sum_{j=1}^{3}U_{4,j,n}$ in the decomposition of $S^{(n)}_{k,2,2}(Y)$ in (\ref{S_decomp}) is asymptotically negligible. 

\vspace{15pt}
{\bf Evaluation of $U_{1,n}$}$\;:\;$ 
By Theorem 3.3.1 and Proposition 15.3.2 in \cite{JP(2012)}, we have $U_{1,n} \stackrel{\mathrm{P}}{\longrightarrow} \sum_{0\leq u\leq1}(\Delta X_{u}^{(1)})^{2}(\Delta X_{u}^{(2)})^{2} \equiv U_{0}$, 
and 
\begin{align}
\sqrt{n}(U_{1,n} - U_{0}) &\overset{\mathcal{L}-s}{\longrightarrow} \text{N}(0, kF_{2,2}). \nonumber
\end{align}

\vspace{15pt}
{\bf Evaluation of $U_{2,1,n}$ and $U_{2,2,n}$}$\;:\;$ 
Let $(Z^{(1)}_{p})$ and $(Z^{(2)}_{p})$ be the mutually independent  sequences of i.i.d. standard normal random variables defined on the second filtered probability space $(\Omega^{(1)}, \mathcal{F}^{(1)}, (\mathcal{F}^{(1)}_{t})_{t \geq 0}, P^{(1)})$, and $(\tau_{p})$ be the co-jump times of the first and second component of the process $(X_{t} = (X_{t}^{(1)}, X_{t}^{(2)}))_{0 \leq t \leq 1}$. We will prove the following result:
\begin{align}
(\sqrt{n}&U_{2,1,n}, \sqrt{n}U_{2,2,n}) \stackrel{\mathcal{L}-s}{\longrightarrow} (U^{(1,2)}_{2}, U^{(1,2)}_{3}), \nonumber \\
U_{2}^{(1,2)} &= \sqrt{8\zeta_{2}}\sum_{p =1}^{\infty}(\Delta X^{(1)}_{\tau_{p}})^{2}(\Delta X^{(2)}_{\tau_{p}})Z^{(2)}_{p}1_{\{\tau_{p}\leq 1\}}, \nonumber \\
U_{3}^{(1,2)} &= \sqrt{8\zeta_{1}}\sum_{p =1}^{\infty}(\Delta X^{(1)}_{\tau_{p}})(\Delta X^{(2)}_{\tau_{p}})^{2}Z^{(1)}_{p}1_{\{\tau_{p}\leq 1\}}. \nonumber 
\end{align}
For the first step of the proof, we prove our result in a simple case when the process $X$ has at most finite jumps in the interval $[0,1]$.  Then we prove the general case when $X$ may have infinite jumps in the interval $[0,1]$. 

\vspace{10pt}
(Step1) : In this step, we introduce some notations. Let $(T_{p})$ be the reordering of the double sequence $(T(m,j): m,j \geq 1)$. A random variable $T(m,j)$ is the successive jump time of the Poisson process $1_{A_{m}\backslash A_{m-1}} \star \mu$ where $A_{m} = \{z:\Gamma(z) > 1/m\}$. Let $\mathcal{P}_{m}$ denote the set of all indices $p$ such that $T_{p} = T(m',j)$ for some $j \geq i$ and some $m' \leq m$. For $(i-1)k\Delta_{n} < T_{p} \leq ik\Delta_{n}$, we define following random variables:
\begin{align}
W^{(k,m)}_{-}(n,p) &= {1 \over \sqrt{k\Delta_{n}}}(X^{(m)}_{T_{p}-} - X^{(m)}_{(i-1)k\Delta_{n}}), \nonumber \\
W^{(k,m)}_{+}(n,p) &= {1 \over \sqrt{k\Delta_{n}}}(X^{(m)}_{ik\Delta_{n}} - X^{(m)}_{T_{p}}), \nonumber \\
W^{(k,m)}(n,p) &= W^{(k,m)}_{-}(n,p) + W^{(k,m)}_{+}(n,p). \nonumber 
\end{align} 
We also set following stochastic processes:
\begin{align}
b(m)_{t} &= b_{t} - \int_{A_{m}\cap \{z:||\delta(t,z)||\leq 1\}}\delta(t,z)\lambda(z),\nonumber \\
X(m)_{t} &= X_{0} + \int_{0}^{t}b(m)_{s}ds + \int_{0}^{t}\sigma_{s}dW_{s} + (\delta 1_{A_{m}^{c}}) \star (\mu-\nu)_{t},\nonumber \\
X'(m) &= X - X(m) = (\delta 1_{A_{m}})\star \mu. \nonumber 
\end{align}
Let $\Omega_{n}(m)$ denote the set of  all $\omega$ such that each interval $[0,1]\cap((i-1)k\Delta_{n}, ik\Delta_{n}]$ contains at most one jump of $X'(m)$, and that $||X(m)(\omega)_{t+s} - X'(m)(\omega)_{t}||\leq 2/m$ for all $t \in [0,1]$ and $s \in [0,\Delta_{n}]$.
Moreover, let  　
\begin{align}
\eta_{p}^{n,k} &=\left(\Delta X^{(1)}_{T_{p}} + \sqrt{k\Delta_{n}}W^{(k,1)}(n,p)\right)^{2}\left(\Delta X^{(2)}_{T_{p}} + \sqrt{k\Delta_{n}}W^{(k,2)}(n,p)\right), \nonumber \\
\widetilde{\eta}_{p}^{n,k} &= \sqrt{\zeta_{2}}\eta_{p}^{n,k}(\Delta v_{i}^{(k,2)}), \nonumber \\
Y_{n}(m) &= \sum_{p \in \mathcal{P}_{m}: T_{p}  k\Delta_{n}[1/k\Delta_{n}]} \widetilde{\eta}_{p}^{n,k}. \nonumber 
\end{align}
By the above notations, on the set $\Omega_{n}(m)$, we have
\begin{align}\label{F00}
\sqrt{n}U_{2,1,n} &= 2\sqrt{n}R^{(n)}_{2,2,1}(X(m)) + 2\sqrt{n}R^{(n)}_{2,2,1}(X'(m)).  
\end{align}

\vspace{10pt}
(Step 2) : In this step, we will prove the stable convergence of $U_{2,1,n}$ and $U_{2,2,n}$. A Taylor expansion of $f(x_{1}, x_{2}) = x_{1}^{2}x_{2}$ yields $\eta_{p}^{n,k} - (\Delta X^{(1)}_{T_{p}})^{2}(\Delta X^{(2)}_{T_{p}}) \stackrel{\mathrm{P}}{\to} 0$. By Proposition 4.4.10 in \cite{JP(2012)}, we have
\begin{align}\label{F0}
\left(\eta_{p}^{n,k}\right)_{p \geq 1} \stackrel{\mathcal{L}-s}{\longrightarrow} \left((\Delta X^{(1)}_{T_{p}})^{2}(\Delta X^{(2)}_{T_{p}})\right)_{p \geq 1}.
\end{align}
The sequence $(\Delta v_{i}^{(k,2)})$ for $k=1,2$ consists of correlated Gaussian random variables with mean $0$ and has the covariance structure
\begin{align}
\text{Cov}(\Delta v_{i}^{(k,2)}, \Delta v_{j}^{(k,2)}) &= 
\begin{cases}
2 & \text{if $i=j$} \\
-1 &\!\!\text{ if $|i-j| = 1$} \\
0 &\text{if $|i-j|>1$}
\end{cases}
.\nonumber 
\end{align}
Using the inequalities (\ref{A1}), (\ref{A3}) and (\ref{A4}), if $|i-j|\geq 1$, then we have $E[|\xi_{i}^{n}\xi_{j}^{n}|] \leq K\Delta_{n}^{3/2}$ where $\xi_{i}^{n} = (\Delta_{i}^{n}X^{(k,1)})^{2}(\Delta_{i}^{n}X^{(k,2)})(\Delta v_{1}^{(k,2)})$. Therefore, the correlation between $\xi_{i}^{n}$ and $\xi_{j}^{n}$ when $|i-j|=1$ is asymptotically negligible.   
Since the set $\{T_{p}:p \in \mathcal{P}_{m}\} \cap [0,1]$ is finite, it follows that
\begin{align}\label{F1}
2\sqrt{n}R^{(n)}_{2,2,1}(X'(m)) \stackrel{\mathcal{L}-s}{\longrightarrow} \sqrt{8\zeta_{2}}\sum_{p \in \mathcal{P}_{m}: T_{p}\leq 1}(\Delta X^{(1)}_{T_{p}})^{2}(\Delta X^{(2)}_{T_{p}})Z^{(2)}_{p} \equiv R_{2,2,1}(X'(m)),  
\end{align}
where $(Z_{p}^{(2)})$ is the sequence of i.i.d. standard normal random variables introduced before (Step 1). 

\vspace{10pt}
(Step 3) : We will prove the joint stable convergence of $(U_{1,n}-U_{0}, U_{2,1,n}, U_{2,2,n})$ in this step. From a slight modification of the proof of Theorem 5.1.2 in \cite{JP(2012)}, it is possible to prove
\begin{align}\label{F2}
R_{2,2,1}(X'(m)) \stackrel{\mathrm{P}}{\to} U_{2}^{(1,2)},
\end{align}
\begin{align}\label{F3}
\lim_{m \to \infty}\limsup_{n \to \infty}\mathrm{P}\left(\Omega_{n}(m)\cap \left\{|\sqrt{n}R^{(n)}_{2,2,1}(X(m))| > \eta\right\} \right) = 0,\quad \text{for all $\eta>0$}. 
\end{align}
Combining these with (\ref{F00}) and (\ref{F1}), and by Proposition 2.2.4 in \cite{JP(2012)}, we obtain $\sqrt{n}U_{2,1,n} \stackrel{\mathcal{L}-s}{\longrightarrow} U^{(2,1)}_{2}$. 
We also have  $\sqrt{n}U_{2,2,n} \overset{\mathcal{L}-s}{\longrightarrow} U_{3}^{(1,2)}$ in the same way.
Since $U_{1,n}$, $U_{2,1,n}$ and $U_{2,2,n}$ are $\mathcal{F}^{(0)}$-conditionally mutually independent by the definition of small noise, we obtain the joint convergence
\begin{align}
\sqrt{n}(U_{1,n}-U_{0}, U_{2,1,n}, U_{2,2,n}) \overset{\mathcal{L}-s}{\longrightarrow} (U_{1,k}^{(1,2)}, U_{2}^{(1,2)}, U_{3}^{(1,2)}). \nonumber
\end{align} 
Therefore, we also obtain $\sqrt{n}(U_{1,n} + U_{2,1,n}+ U_{2,2,n}-U_{0})\overset{\mathcal{L}-s}{\longrightarrow} U_{1,k}^{(1,2)} + U_{2}^{(1,2)} + U_{3}^{(1,2)}$.

\vspace{15pt}

{\bf Evaluation of the remaining terms}$\;:\;$ 
We shall prove that the remaining terms $\sum_{j=1}^{3}U_{3,j,n}$ and $\sum_{j=1}^{3}U_{4,j,n}$ in the decomposition of $S^{(n)}_{k,2,2}$ in (\ref{S_decomp}) are asymptotically negligible. Let $\xi_{i}^{n} = \epsilon_{2,n}^{2}(\Delta_{i}^{n}X^{(k,1)})^{2}(\Delta v_{i}^{(k,2)})^{2}$. Then 
$U_{3,1,n} = \sum_{i=1}^{[n/k]}\xi_{i}^{n}$. 
Using the inequalities (\ref{A1}), (\ref{A3}) and (\ref{A4}), we have
\begin{align}
E[|\xi_{i}^{n}|] = \epsilon_{2,n}^{2}E\left[(\Delta v_{i}^{(k,2)})^{2}\right]E[(\Delta_{i}^{n}X^{(k,1)})^{2}] \leq K\Delta_{n}^{2}. \nonumber 
\end{align}
Therefore, we obtain $U_{3,1,n} = O_{P}(1/n)$. Similarly, we have $U_{3,2,n} = U_{3,3,n} = U_{4,1,n} = U_{4,2,n} = O_{P}(1/n)$.
Since $n^{-1}\sum_{i=1}^{[n/k]}(\Delta v_{i}^{(k,1)})^{2}(\Delta v_{i}^{(k,2)})^{2} \stackrel{\mathrm{P}}{\to} 4\zeta_{1}\zeta_{2}$, we have $U_{4,3,n} = O_{P}(1/n\sqrt{n})$. 
Consequently,  we obtain $\sqrt{n}(S_{k,2,2}^{(n)}(Y) - U_{0}) \overset{\mathcal{L}-s}{\longrightarrow} U_{1,k}^{(1,2)} + U_{2}^{(1,2)} + U_{3}^{(1,2)}$.
\end{proof}

\vspace{10pt}

\begin{proof}[Proof of Theorem 2]

We first prove a technical tool for the proof of Theorem 2. The result of Theorem 2 follows immediately from the following proposition. 
\begin{proposition}\label{Prp4}
Suppose Assumptions 1, 2 and 3 are satisfied with $q=0$ in Assumption 2. Let 
\begin{align}
S_{0} &= \sum_{0\leq s\leq1}(\Delta X_{s}^{(1)})^{2}(\Delta X_{s}^{(2)})^{2},\nonumber  \\ 
W^{(n)} &= \sqrt{n}[(S_{2,2,2}^{(n)}(Y) - S_{0}) - (S_{1,2,2}^{(n)}(Y) - S_{0})]. \nonumber 
\end{align}
Then, in restriction to the set $\Omega^{(j)}$, we have 
\begin{align}
W^{(n)} \overset{\mathcal{L}-s}{\longrightarrow} U^{(1,2)} = U_{1,1}^{(1,2)} + U_{2}^{(1,2)} + U_{3}^{(1,2)}. \nonumber 
\end{align}
%
\end{proposition}

\vspace{10pt}
\begin{proof}[Proof of Proposition \ref{Prp4}] 

By using the inequalities (\ref{A1}), (\ref{A3}) and (\ref{A4}), we have
\begin{align}
&\sum_{i=1}^{[n/2]}(\Delta_{i}^{n}Y^{(2,1)})^{2}(\Delta_{i}^{n}Y^{(2,2)})^{2}\nonumber \\
&\quad =\sum_{i=1}^{[n/2]}(\Delta_{i}^{n}X^{(2,1)})^{2}(\Delta_{i}^{n}X^{(2,2)})^{2} + 2R^{(n)}_{2,1,2}(X) + 2R^{(n)}_{2,2,1}(X) +O_{P}(1/\sqrt{n}),
\nonumber 
\end{align}
where $R_{k,r,s,}^{(n)}(X)$ are definded in (\ref{Rkrs}). 
We decompose $W^{(n)}$ into three leading terms:
\begin{align}
W^{(n)} &= \sqrt{n}(S_{2,2,2}^{(n)}(X) - S_{1,2,2}^{(n)}(X)) + 2\sqrt{n}(R_{2,1,2}^{(n)}(X) - R_{1,1,2}^{(n)}(X)) \nonumber  \\
 &\quad + 2\sqrt{n}(R_{2,2,1}^{(n)}(X) - R_{1,2,1}^{(n)}(X))  + o_{P}(1) \nonumber \\
&= \text{I}_{n,1}+ \text{II}_{n,2} + \text{III}_{n,3} + o_{P}(1), \nonumber
\end{align}
Because of the independence of the noise $v_{i}^{(1)}$ and $v_{i}^{(2)}$, the three terms $\text{I}_{n,1}$, $\text{II}_{n,1}$ and $\text{III}_{n,1}$ are asymptotically mutually independent. Therefore, it suffices to evaluate each term. We can rewrite $\text{II}_{n,2}$ by using the estimation inequalities (\ref{A1}), (\ref{A3}) and (\ref{A4}),
\begin{align}
\text{II}_{n,2} &= 2\sqrt{\zeta_{1}}\sum_{i=1}^{[n/2]}\sum_{j=2(i-1)+1}^{2i}(\Delta_{j}^{n}X^{(1)})(\Delta_{j}^{n}X^{(2)})^{2}\left(\sum_{l=2(i-1)+1}^{2i}\Delta v_{l}^{(1)} - \Delta v_{j}^{(1)}\right) + O_{P}(1/\sqrt{n}) \nonumber  \\
&= \widetilde{\text{II}}_{n,2} + O_{P}(1/\sqrt{n}). \nonumber 
\end{align}
By the application of Proposition 15.3.2 in \cite{JP(2012)} to $\text{I}_{n,1}$ and the similar argument of the evaluation of $U_{2,1,n}$ in the proof of Theorem 1, we obtain that 
$W^{(n)}$ converge stably to $U_{1}^{(1,2)} + U_{2}^{(1,2)} + U_{3}^{(1,2)}$. \end{proof}
Finally, we prove the stable convergence of $T^{(n)}$. 
From the definition of $T^{(n)}$, we have
\begin{align}
T^{(n)}- 1 = {S_{2,2,2}^{(n)}(Y) - S_{1,2,2}^{(n)}(Y) \over S_{1,2,2}^{(n)}(Y)}. \nonumber
\end{align}
Then by Proposition \ref{Prp4}, we obtain the desired result. 
\end{proof}

\begin{proof}[Proof of Proposition \ref{T_SP}]
Since the asymptotic size of the test statistic $T^{(n)}$ follows from Theorem \ref{T_Asy}, we prove the consistency of the test here. 
To describe the limit variable of $T^{(n)}$ under the alternative hypothesis $\mathbb{H}_{1}$, we introduce some notations. We use some notations in \cite{JT(2009)}. 
\begin{itemize}
\item a sequence $(\kappa_{p})$ of uniform variables on $[0,1]$.
\item a sequence $(L_{p})$ of uniform variables on $\{0,1\}$, that is, $\mathrm{P}(L_{p} = 0) = \mathrm{P}(L_{p} = 1) = 1/2$. 
\item four sequences $(U_{p})$, $(U'_{p})$, $(\widetilde{U}_{p})$, $(\widetilde{U}'_{p})$ of two-dimensional $\text{N}(0, I_{2})$ variables.
\item two sequences $(\overline{Z}_{p},\overline{Z}''_{p})$,  $(\widetilde{Z}_{p},\widetilde{Z}''_{p})$ of two-dimensional 
$\text{N}
\left(
\left(
\begin{array}{c}
0 \\
0
\end{array}
\right),
\left(
\begin{array}{cc}
2 & -1 \\
-1 & 2
\end{array}
\right)
\right)$ variables.
\end{itemize}
The variables introduced above are defined on $(\Omega^{(1)}, \mathcal{F}^{(1)}, (\mathcal{F}^{(1)}_{t})_{t \geq 0}, P^{(1)})$. Then we define two dimensional variables:
\begin{align}
R_{p} &= \sigma_{T_{p}}\left(\sqrt{\kappa_{q}}U_{p} + \sqrt{1-\kappa_{p}}U'_{p}\right), \nonumber \\
R'_{p} &= \sigma_{T_{p}}\left(\sqrt{L_{q}}\widetilde{U}_{p} + \sqrt{1-L_{p}}\widetilde{U}'_{p}\right), \nonumber \\
R''_{p} &= R_{p} + R'_{p}. \nonumber 
\end{align}
We also define following variables:
\begin{align}
\widetilde{D}'' &= \sum_{p: T_{p}\leq 1}\left((\Delta X_{T_{p}}^{(1)}R_{p}^{''(2)})^{2} +(\Delta X_{T_{p}}^{(2)}R_{p}^{''(1)})^{2} \right), \nonumber \\
\widetilde{D} &= \sum_{p: T_{p}\leq 1}\left((\Delta X_{T_{p}}^{(1)}R_{p}^{(2)})^{2} +(\Delta X_{T_{p}}^{(2)}R_{p}^{(1)})^{2} \right), \nonumber \\
L'' &= \sum_{p: T_{p}\leq 1}\left(\sqrt{\zeta_{2}}(\Delta X_{T_{p}}^{(1)})^{2}R_{p}^{''(2)}\widetilde{Z}''_{p} +\sqrt{\zeta_{1}}(\Delta X_{T_{p}}^{(2)})^{2}R_{p}^{''(1)}\overline{Z}''_{p} \right), \nonumber \\
L &= \sum_{p: T_{p}\leq 1}\left(\sqrt{\zeta_{2}}(\Delta X_{T_{p}}^{(1)})^{2}R_{p}^{(2)}\widetilde{Z}_{p} +\sqrt{\zeta_{1}}(\Delta X_{T_{p}}^{(2)})^{2}R_{p}^{(1)}\overline{Z}_{p} \right), \nonumber \\
H &= \int_{0}^{1}\left(c_{s}^{(1,1)}c_{s}^{(2,2)} + 2(c_{s}^{(1,2)})^{2}\right)ds. \nonumber 
\end{align}
By using (\ref{A1}), (\ref{A3}) and (\ref{A4}), in  restriction to the set $\Omega^{(d)}$, we have
\begin{align}
S_{k,2,2}^{(n)}(Y) &= S_{k,2,2}^{(n)}(X) + 2R_{k,2,1}^{(n)}(X) + 2R_{k,1,2}^{(n)}(X) +o_{P}(1/n), \nonumber 
\end{align}
where $R_{k,r,s}^{(n)}(X)$ is defined in (\ref{Rkrs}). 
Moreover, from similar argument of the proof of Theorem 3.1 in \cite{JT(2009)}, in restriction to the set $\Omega^{(d)}$, it is possible to have
\begin{align}
n\times (S^{(n)}_{2,2,2}(X), S^{(n)}_{1,2,2}(X), R^{(n)}_{2,2,1}(X)+R^{(n)}_{2,1,2}(X), R^{(n)}_{1,2,1}(X)+R^{(n)}_{1,1,2}(X)) 
&\stackrel{\mathcal{L}-s}{\longrightarrow} (\widetilde{D}''+2H, \widetilde{D}+H, L'', L). \nonumber 
\end{align}
Therefore, we obtain
\begin{align}
T^{(n)} & 
\stackrel{\mathcal{L}-s}{\longrightarrow} \Phi = {\widetilde{D}'' + 2H + 2L''\over \widetilde{D} + H + 2L}.\nonumber 
\end{align}
Since $(\widetilde{D}'', \widetilde{D}, L'', L)$ admits a density conditionally on $\mathcal{F}$ and on being in $\Omega^{(d)}$, $T^{(n)} \neq 1$ a.s. on $\Omega^{(d)}$.
\end{proof}

\subsection{Proofs for Section 5} 

\begin{proof}[Proof of Proposition \ref{Prp2}]

For the consistency of the SIML, see the proof of Theorem 1 in \cite{KS(2013)}.
\end{proof}

\begin{proof}[Proof of Proposition \ref{Prp3}]

We only give the proof of a consistency of $\widehat{D}^{(l,m)}_{p,q}(1,1)$ and $\widehat{A}^{(l,m)}(r)$. The proofs of consistency of the other estimators are similar.
Under the assumptions of Proposition \ref{Prp3}, for any $\eta>0$, $\alpha>0$, $\theta \in (0, 1/2)$, we have
\begin{align}
&\mathrm{P}\left(\sup_{1\leq i \leq n}1_{\{|\epsilon_{j,n}\Delta v_{i}^{(j)}| > \alpha \Delta_{n}^{\theta}\}} > \eta\right)\nonumber \\ 
&\quad \leq  \mathrm{P}\left(\sum_{i=1}^{n}1_{\{|\epsilon_{j,n}\Delta v_{i}^{(j)}| > \alpha \Delta_{n}^{\theta}\}} > \eta\right) \leq \mathrm{P}\left(\bigcup_{i=1}^{n}\left\{\epsilon_{j,n}|\Delta v_{i}^{(j)}| > \alpha \Delta_{n}^{\theta}\right\}\right) \nonumber \\
&\quad \leq \sum_{i=1}^{n}\mathrm{P}\left(\epsilon_{j,n}|\Delta v_{i}^{(j)}| > \alpha \Delta_{n}^{\theta}\right) \leq 4\sum_{i=1}^{n}\max\left\{\mathrm{P}\left(\epsilon_{j,n}v_{i}^{(j)} > {\alpha \Delta_{n}^{\theta} \over 2}\right), \mathrm{P}\left(\epsilon_{j,n}v_{i}^{(j)} < -{\alpha \Delta_{n}^{\theta} \over 2}\right)\right\} \nonumber \\
&\quad= o(1). \nonumber 
\end{align}
Then $\sup_{1\leq i \leq n}\left|1_{\{||\Delta_{i}^{n}Y||> \alpha \Delta_{n}^{\theta}\}} - 1_{\{||\Delta_{i}^{n}X||> \alpha \Delta_{n}^{\theta}/2\}}\right| \stackrel{\mathrm{P}}{\to} 0. 
$ Therefore, we obtain
\begin{align}
\sum_{i=1}^{n}(\Delta_{i}^{n}Y^{(1)})(\Delta_{i}^{n}Y^{(2)})1_{\{||\Delta_{i}^{n}Y||> \alpha \Delta_{n}^{\theta}\}} - \sum_{i=1}^{n}(\Delta_{i}^{n}Y^{(1)})(\Delta_{i}^{n}Y^{(2)})1_{\{||\Delta_{i}^{n}X||> \alpha \Delta_{n}^{\theta}/2\}} \stackrel{\mathrm{P}}{\to} 0. \nonumber 
\end{align}
Moreover, for $\eta>0$,  
\begin{align}
&\mathrm{P}\left(\left|\epsilon_{1,n}\epsilon_{2,n}\sum_{i=1}^{n}
(\Delta_{i}^{n}X^{(1)})(\Delta_{i}^{n}X^{(2)})(\Delta v_{i}^{(1)})(\Delta v_{i}^{(2)})1_{\{||\Delta_{i}^{n}X||> \alpha \Delta_{n}^{\theta}\}}\right|> \eta\right) \nonumber \\
&\quad \leq {
(\epsilon_{1,n}\epsilon_{2,n})^{2} \over \eta^{2}}E\left[\left(\sum_{i=1}^{n}(\Delta_{i}^{n}X^{(1)})(\Delta_{i}^{n}X^{(2)})1_{\{||\Delta_{i}^{n}X||> \alpha \Delta_{n}^{\theta}\}}(\Delta v_{i}^{(1)})(\Delta v_{i}^{(2)})\right)^{2}\right] \nonumber \\
&\quad \leq {
(\epsilon_{1,n}\epsilon_{2,n})^{2} \over \eta^{2}}\left(\sum_{i=1}^{n}A_{i}^{1,n} + \sum_{i=2}^{n}A_{i}^{2,n} + \sum_{i=1}^{n-1}A_{i}^{3,n}\right), \nonumber 
\end{align}
where 
\begin{align}
A_{i}^{1,n} &= E\left[(\Delta_{i}^{n}X^{(1)})^{2}(\Delta_{i}^{n}X^{(2)})^{2}(\Delta v_{i}^{(2)})^{2}(\Delta v_{i}^{(1)})^{2}\right], \nonumber \\
A_{i}^{2,n} &= E\left[(\Delta_{i}^{n}X^{(1)})(\Delta_{i-1}^{n}X^{(1)})(\Delta_{i}^{n}X^{(2)})(\Delta_{i-1}^{n}X^{(2)})(\Delta v_{i}^{(1)})(\Delta v_{i-1}^{(1)})(\Delta v_{i}^{(2)})(\Delta v_{i-1}^{(2)})\right],\nonumber \\
A_{i}^{3,n} &= E\left[(\Delta_{i+1}^{n}X^{(1)})(\Delta_{i}^{n}X^{(1)})(\Delta_{i+1}^{n}X^{(2)})(\Delta_{i}^{n}X^{(2)})(\Delta v_{i+1}^{(1)})(\Delta v_{i}^{(1)})(\Delta v_{i+1}^{(2)})(\Delta v_{i}^{(2)})\right]. \nonumber
\end{align}

Now we evaluate the last three terms in the above inequality. By H\"older's inequality and the inequalities of It\^o semimartingale, we have $A_{i}^{1,n} = A_{i}^{2,n}  = A_{i}^{3,n} = O(n^{-2})$. 
Therefore,
\begin{align}
\epsilon_{1,n}\epsilon_{2,n}\sum_{i=1}^{n}
(\Delta_{i}^{n}X^{(1)})(\Delta_{i}^{n}X^{(2)})(\Delta_{i}v^{(1)})(\Delta_{i}v^{(2)})1_{\{||\Delta_{i}^{n}X||> \alpha \Delta_{n}^{\theta}\}} \stackrel{\mathrm{P}}{\longrightarrow} 0. \nonumber 
\end{align}
From the similar argument, for $1 \leq l,m \leq 2$, we have 
\begin{align}
\epsilon_{m,n}^{2}\sum_{i=1}^{n}
(\Delta_{i}^{n}X^{(l)})^{2}(\Delta_{i}v^{(m)})^{2}1_{\{||\Delta_{i}^{n}X||> \alpha \Delta_{n}^{\theta}\}} &\stackrel{\mathrm{P}}{\longrightarrow} 0, \nonumber \\
\epsilon_{m,n}\sum_{i=1}^{n}
(\Delta_{i}^{n}X^{(l)})^{2}(\Delta_{i}^{n}X^{(m)})(\Delta_{i}v^{(m)})1_{\{||\Delta_{i}^{n}X||> \alpha \Delta_{n}^{\theta}\}} &\stackrel{\mathrm{P}}{\longrightarrow} 0, \nonumber \\
\epsilon_{l,n}\epsilon_{m,n}^{2}\sum_{i=1}^{n}
(\Delta_{i}^{n}X^{(l)})(\Delta_{i}v^{(l)})(\Delta_{i}v^{(m)})^{2}1_{\{||\Delta_{i}^{n}X||> \alpha \Delta_{n}^{\theta}\}} &\stackrel{\mathrm{P}}{\longrightarrow} 0. \nonumber
\end{align}
Hence we have 
\begin{align}
\sum_{i=1}^{n}(\Delta_{i}^{n}Y^{(1)})(\Delta_{i}^{n}Y^{(2)})1_{\{||\Delta_{i}^{n}Y||>\alpha \Delta_{n}^{\theta}\}} - \sum_{i=1}^{n}(\Delta_{i}^{n}X^{(1)})(\Delta_{i}^{n}X^{(2)})1_{\{||\Delta_{i}^{n}X||>\alpha \Delta_{n}^{\theta}/2\}} \stackrel{\mathrm{P}}{\to} 0. \nonumber 
\end{align}
Then $\sum_{i=1}^{n}(\Delta_{i}^{n}Y^{(1)})(\Delta_{i}^{n}Y^{(2)})1_{\{||\Delta_{i}^{n}Y||>\alpha \Delta_{n}^{\theta}\}} \stackrel{\mathrm{P}}{\to} \sum_{0 \leq s \leq 1}(\Delta X_{s}^{(1)})(\Delta X_{s}^{(2)}).$ Therefore, from Theorem 9.4.1 and Theorem 9.5.1 in \cite{JP(2012)}, we have 
\begin{align}
\sum_{i=1}^{n}\widehat{c}^{(1,2)}_{i}(\Delta_{i}^{n}Y^{(1)})(\Delta_{i}^{n}Y^{(2)})&1_{\{||\Delta_{i}^{n}Y||>\alpha \Delta_{n}^{\theta}\}} \stackrel{\mathrm{P}}{\to} D^{(1,2)}_{1,2}(1,1),\nonumber  \\
\Delta_{n}\sum_{i=1}^{n-k_{n}+1}(\widehat{c}^{(1,2)}_{i})^{r} &\stackrel{\mathrm{P}}{\to}A^{(1,2)}(r). \nonumber
\end{align}
Therefore we obtain the desired results.
\end{proof}

\begin{proof}[Proofs of Corollaries 1 and 2]

The result immediately follows from the remarks in Section 3,  Theorem 2, Propositions \ref{Prp2} and \ref{Prp3}. 
\end{proof}


\section{Figures and tables} 

\begin{figure}[H]
  \begin{center}
    \includegraphics[clip, width=9.0cm]{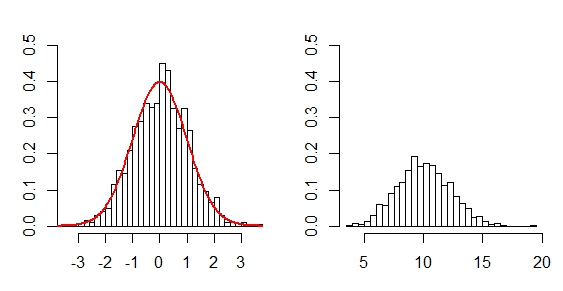}
       \caption{Empirical distributions of RV when $dX_{t} = \sigma_{t}dW_{t}$, $\zeta=10^{-2}$. The left figure corresponds to the bias-variance corrected case implied by Corollary 1(i). The right figure corresponds to the no correction case (statistics are standardized as $\zeta_{1}=0$ in Corollary 1(i)). The red line is the density of a standard normal distribution.} 
       \label{fig:B1}
  \end{center}
\end{figure}

\vspace{-30pt}
\begin{figure}[H]
  \begin{center}
    \includegraphics[clip, width=9.0cm]{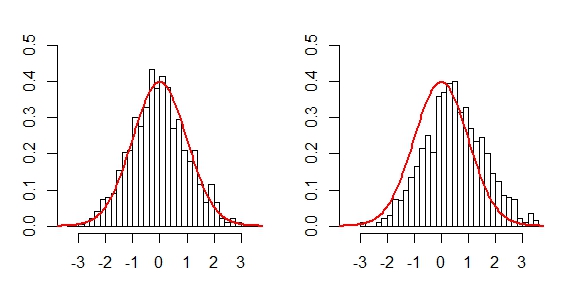}
       \caption{Empirical distributions of RV when $dX_{t} = \sigma_{t}dW_{t} + dJ^{cp}_{t}$, $\zeta = 10^{-2}$. We plot the left and right figures in the same way as Figure 2.} 
       \label{fig:B2}
  \end{center}
\end{figure}


\vspace{-30pt}
\begin{figure}[H]
\begin{center}
\includegraphics[clip, scale=0.4, bb=0 0 400 400]{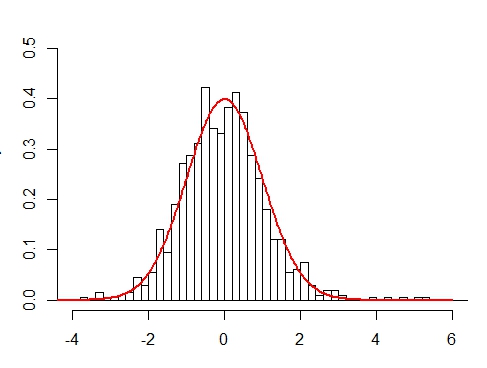}
\caption{Empirical distribution of the co-jump test implied by Corollary 2 in C1-IV. The red line is the density of a standard normal distribution.} 
\label{fig:A}
\end{center}
\end{figure}

\vspace{-20pt}
\begin{table}[H]
\begin{center}
\begin{tabular}{ccccccc}
\hline
Case
 &  \multicolumn{1}{c}{CJ1-\text{(i)}} & \multicolumn{1}{c}{CJ1-\text{(ii)}} & \multicolumn{1}{c}{CJ1-\text{(iii)}}
&  \multicolumn{1}{c}{CJ2-\text{(i)}} & \multicolumn{1}{c}{CJ2-\text{(ii)}} & \multicolumn{1}{c}{CJ2-\text{(iii)}} \\
\hline
RMSE &2.137 & 2.192 & 2.349 
& 1.701  & 1.735 & 1.878 \\ 
\hline \hline
Case
 &  \multicolumn{1}{c}{SJ1-\text{(i)}} & \multicolumn{1}{c}{SJ1-\text{(ii)}} & \multicolumn{1}{c}{SJ1-\text{(iii)}} 
&  \multicolumn{1}{c}{SJ2-\text{(i)}} & \multicolumn{1}{c}{SJ2-\text{(ii)}} & \multicolumn{1}{c}{SJ2-\text{(iii)}} \\
\hline
RMSE & 11.45 & 11.47 & 11.40 
& 9.501  & 9.496 & 9.470  \\ 
\hline
\end{tabular}
\caption{RMSEs of $\widehat{\text{IV}}$. Values are reported as multiples of $10^{-3}$.}
\end{center}
\end{table}

\vspace{-20pt}
\begin{table}[H]
\begin{center}
\begin{tabular}{ccccccccc}
\hline
Case
 &  \multicolumn{1}{c}{C1-\text{I}} & \multicolumn{1}{c}{C1-\text{II}} & \multicolumn{1}{c}{C1-\text{III}} & \multicolumn{1}{c}{C1-\text{IV}} 
&  \multicolumn{1}{c}{C2-\text{I}} & \multicolumn{1}{c}{C2-\text{II}} & \multicolumn{1}{c}{C2-\text{III}} & \multicolumn{1}{c}{C2-\text{IV}} \\
\hline
Size 
&0.068  & 0.064 &0.065 &0.071 
&0.058  &0.053 &0.052 &0.046 \\ 
&0.082  & 0.078 &0.139 &0.253 
&0.069  &0.070 &0.128 &0.237 \\
  
\hline
\hline
Case
&  \multicolumn{1}{c}{D1-\text{I}} & \multicolumn{1}{c}{D1-\text{II}} & \multicolumn{1}{c}{D1-\text{III}} & \multicolumn{1}{c}{D1-\text{IV}} 
&  \multicolumn{1}{c}{D2-\text{I}} & \multicolumn{1}{c}{D2-\text{II}} & \multicolumn{1}{c}{D2-\text{III}} & \multicolumn{1}{c}{D2-\text{IV}} \\
\hline 
 Power  
&0.989 &0.992  &0.990 &0.982 
&0.995 &0.996  &0.995 &0.988  \\
&0.986 &0.994  &0.989 &0.985 
&0.987 &0.984  &0.986 &0.990  \\
\hline

\end{tabular}
\caption{Empirical size and power of the co-jump test ($5\%$ significant level) are reported for the proposed test (top), and the test proposed in \cite{JT(2009)} (bottom).}
\end{center}
\end{table}

\section*{Acknowlegement}
The author would like to thank Professor Naoto Kunitomo and Associate Professor Kengo Kato for their suggestions and encouragements, as well as Professor Yasuhiro Omori for his constructive comments.

\end{document}